\documentclass[a4paper,14pt]{article}

\usepackage{mathtools, nicematrix}
\usepackage[left=20mm, top=20mm, right=20mm, bottom=20mm, nohead]{geometry}
\usepackage[T1,T2A]{fontenc}
\usepackage[utf8]{inputenc}
\usepackage[english]{babel}
\usepackage{amsmath}
\usepackage{amsfonts,amssymb}
\usepackage{amsthm}
\usepackage[electronic]{ifsym}
\usepackage{authblk}
\usepackage{hyperref}
\usepackage[nameinlink, capitalize]{cleveref}
\usepackage[shortlabels]{enumitem}

\theoremstyle{theorem}
\newtheorem{theorem}{Theorem}[section]

\theoremstyle{definition}
\newtheorem{definition}[theorem]{Definition}
  
\theoremstyle{lemma}
\newtheorem{lemma}[theorem]{Lemma}

\theoremstyle{proposition}
\newtheorem{proposition}[theorem]{Proposition}

\theoremstyle{remark}
\newtheorem*{remark}{Remark}

\theoremstyle{corollary}
\newtheorem{corollary}[theorem]{Corollary}

\theoremstyle{conjecture}

\theoremstyle{definition}

\theoremstyle{question}
\newtheorem{question}[theorem]{Question}

\newcommand{\R}{\mathbb{R}}
\newcommand{\N}{\mathbb{N}}

\newcommand{\Z}{\mathbb{Z}}

\newcommand{\ExpField}{\mathsf{ExpField}}
\newcommand{\supp}{\mathrm{supp}}
\newcommand{\OR}{\mathrm{OR}}
\newcommand{\OSR}{\mathrm{OSR}}
\newcommand{\Th}{\mathrm{Th}}
\newcommand{\LE}{\mathrm{LE}}
\newcommand{\IP}{\mathrm{IP}}
\newcommand{\Ord}{\mathrm{Ord}}
\newcommand{\ord}{\mathrm{ord}}
\newcommand{\id}{\mathrm{id}}

\begin{document}
\title{Analogues of Shepherdson's Theorem for a language with exponentiation}
\author{Konstantin Kovalyov}
\affil{
Steklov Mathematical Institute of Russian Academy of Sciences, Moscow, Russia

\href{mailto:kovalyov.ka@mi-ras.ru}{kovalyov.ka@mi-ras.ru}
}

\maketitle

\begin{abstract}
    In 1964 J. Shepherdson \cite{shepherdson:1964} proved that a discretely ordered semiring $\mathcal{M}^+$ satisfies $\sf{IOpen}$ (quantifier-free induction) iff the corresponding ring $\mathcal{M}$ is an integer part of a model of the theory of real closed fields. In this paper, we consider open induction schema in the language of arithmetic expanded by exponentiation or by the power function and try to find similar criteria for models of these theories.
    
    For several recursively axiomatized extensions $T$ of the theory of real closed fields we obtain analogues of Shepherdson's Theorem in the following sense: If an exponential field $\mathcal R$ is a model of $T$ and a discretely ordered ring $\mathcal M$ is an exponential integer part of $\mathcal R$, then $\mathcal M^+$ is a model of open induction in the expanded language. The proof of the opposite implication --- that for any model $\mathcal M$ of open induction in the expanded language there exists an exponential field $\mathcal R \vDash T$ such that $\mathcal M$ is an exponential integer part of $\mathcal R$ --- remains, in general, an open question. However, we isolate a natural sufficient condition, related to the well-known Bernoulli inequality, under which this result holds. We define a finite extension $T$ of the usual open induction so that, for any discretely ordered ring $\mathcal M$, the semiring $\mathcal M^+$ satisfies $T$ iff there is an exponential real closed field $\mathcal R$ with the inequality $\exp(x) \geqslant 1 + x$ such that $\mathcal M$ is an exponential integer part of $\mathcal R$. Using these results, we obtain some concrete independence results for these theories.
\end{abstract}

\section{Introduction}\label{section_introduction}

In 1964, J. Shepherdson proved a remarkable theorem relating models of weak arithmetic satisfying quantifier-free induction to models of real closed fields. This connection allows one to build computable models of weak arithmetical theories from models of the theory of real closed fields, which is complete and decidable (see, for instance, \cite{shepherdson:1964, berarducci}). Thus, Shepherdson's Theorem provides a lower bound to the Tennenbaum phenomenon: no nonstandard model of a sufficiently strong arithmetical theory can have computable addition or multiplication functions. (Currently, Tennenbaum's Theorem is known for theories as weak as $\mathsf{IE}_1$ of bounded existential induction, see \cite{wilmers_1985}.)

In this paper we approach the Tennenbaum boundary from below and are looking for extensions of Shepherdson's Theorem to stronger fragments of arithmetic (such examples exist; as shown in \cite{berarducci}, quantifier-free induction with the normality axiom admits a nonstandard recursive model). Over the years, a lot of attention has been given to exponential fields, especially in connection with the well-known open problem posed by Tarski \cite{tarski1948}: Does the expansion of the field of reals by an exponential function have a decidable first-order theory? This problem led to breakthrough results in the model theory of exponential fields, including the results by A. Wilkie, L. van den Dries and A. Macintyre, some of which will be commented in more details below. (Currently, a positive answer to Tarski's problem is only known modulo Schanuel's conjecture, see~\cite{Macintyre1996OnTD}.)

Analogues of Shepherdson's Theorem for the language of arithmetic expanded by an exponential function have been considered in a few papers, in particular in \cite{ITA_2008__42_1_105_0, jerabek2022, jerabek2024}. The main problem in this study is to develop the techniques of building models of open induction in the language of arithmetic with exponentiation and relating them in a meaningful way to certain classes of exponential fields. Our contributions in this paper are as follows. Firstly, we prove that exponential integer parts of exponential fields satisfying the extreme value theorem scheme for exponential terms are models of the quantifier-free induction in the arithmetic language with exponentiation (see \autoref{theorem_M_exp_ip_ExpField+MaxVal}). Secondly, $x^y$-integer parts of exponential fields satisfying a certain recursively axiomatized theory, true in $(\R, \exp)$, are models of the quantifier-free induction in the arithmetic language with $x^y$ (see \autoref{theorem_(M, x^y)_ip_R_exp}). Thirdly, $x^y$-integer parts of real closed exponential fields satisfying $\exp(x) \geqslant 1 + x$ are exactly the models of the quantifier-free induction in the pure arithmetic language extended with a certain finite theory expressing basic properties of $x^y$ including the Bernoulli inequality (see \autoref{theorem_(M,x^y)_is_iopen+T_x^y<=>ip_of_ExpField_RCF+Bern)}). Finally, we build a nonstandard model of quantifier-free induction in the arithmetic language with $x^y$ to show some concrete independence results (see \cref{corollary_independence_result_1} and \cref{corollary_independence_result_2}). Namely, the irrationality of $\sqrt{2}$ and Fermat's Last Theorem for $n = 3$ are independent from the quantifier-free induction in the arithmetic language with $x^y$. In the remaining part of the introduction we describe these results and their context in more detail.


A discretely ordered ring $\mathcal M$ is called an \emph{integer part} of an ordered ring $\mathcal R \supseteq \mathcal M$ if for all $r \in R$ there exists an $m \in M$ such that $m \leqslant r < m + 1$. By $\mathsf{IOpen}$ we denote the theory of discretely ordered semirings with the induction scheme for quantifier-free formulas (also called open formulas) in the language $\mathcal L_{\OSR} := (+, \cdot, 0, 1, \leqslant)$ of ordered semirings (which will serve us as the basic arithmetic language). Shepherdson's result is as follows. For a discretely ordered ring $\mathcal M$, a discretely ordered semiring $\mathcal{M}^+$ satisfies $\mathsf{IOpen}$ iff the real closure of the quotient field of $\mathcal{M}$ contains $\mathcal{M}$ as an integer part (here $\mathcal M^+$ is the nonnegative part of $\mathcal M$). The theory $\mathsf{IOpen}$ and its models were extensively studied, see, for instance, \cite{berarducci,wilkie1978, van_den_dries1980,van_den_dries1980_weak_arithm,otero1990, mourgues:1993, glivicka2017shepherdsons} and other papers. M. H. Mourgues and J.-P. Ressayre \cite{mourgues:1993} showed that every real closed field has an integer part. This important result was generalized by J.-P. Ressayre~\cite{ressayre:1993} (see also \cite{PaolaDAquino2012}) to real closed exponential fields with growth axiom for exponentiation (RCEF for short), namely, every RCEF has an exponential integer part (i.e., an integer part such that its nonnegative part is closed under exponentiation). 

In this paper we consider a theory $\mathsf{IOpen(\exp)}$ in the language $\mathcal{L}_{\OSR}(\exp)$, a theory $\mathsf{IOpen(x^y)}$ and a finite set of natural axioms for power function $\mathsf{T_{x^y}}$ in the language $\mathcal{L}_{\OSR}(x^y)$. $\mathsf{IOpen}(\exp)$ and $\mathsf{IOpen(x^y)}$ are axiomatized by quantifier-free induction schemata in the corresponding languages and some basic axioms for $\exp$ and $x^y$. In \cref{section_ip_of_exp_fields} and \cref{section_construction_of_K_M} we establish some analogues of Shepherdson's Theorem for models of $\mathsf{IOpen(\exp)}$, $\mathsf{IOpen(x^y)}$ and $\mathsf{IOpen} + \mathsf{T_{x^y}}$ in terms of ``exponential'' integer parts. Namely, we obtain the following:
\begin{itemize}
    \item[(Thm. \ref{theorem_M_exp_ip_ExpField+MaxVal})] every exponential integer part (i.e. an integer part, whose set of positive elements is closed under $\exp$) of a model of $\mathsf{ExpField + MaxVal}(\mathcal L_{\OR}(\exp)) + \exp(1) = 2$ is a model of $\mathsf{IOpen}(\exp)$;
    \item[(Thm. \ref{theorem_(M, x^y)_ip_R_exp})] every $x^y$-integer part (i.e. an integer part, whose set of positive elements is closed under $x^y = \exp(y \log(x))$) of a model of $\mathsf{KTB}$ (some concrete recursive subtheory of $\Th(\R, \exp)$) is a model of $\mathsf{IOpen(x^y)}$;
    \item[(Thm. \ref{theorem_(M,x^y)_is_iopen+T_x^y<=>ip_of_ExpField_RCF+Bern)})] $x^y$-integer parts of models of $\mathsf{ExpField + RCF + } \forall x(\exp(x) \geqslant 1 + x)$ are exactly models of $\mathsf{IOpen} + \mathsf{T_{x^y}}$.
\end{itemize}
Here $\mathsf{RCF}$ denotes the theory of real closed fields in the language $\mathcal L_{\OR} := (+, -, \cdot, 0, 1, \leqslant)$ of ordered rings, $\mathsf{ExpField}$ denotes the theory of exponential fields and $\mathsf{MaxVal}(\mathcal L_{\OR}(\exp))$ denotes the schema of extreme value theorems for all $\mathcal L_{\OR}(\exp)$-terms. The proofs of \autoref{theorem_M_exp_ip_ExpField+MaxVal}, \autoref{theorem_(M, x^y)_ip_R_exp}, and the forward implication of \autoref{theorem_(M,x^y)_is_iopen+T_x^y<=>ip_of_ExpField_RCF+Bern)} are presented in \cref{section_ip_of_exp_fields}; the opposite implication of \autoref{theorem_(M,x^y)_is_iopen+T_x^y<=>ip_of_ExpField_RCF+Bern)} is treated in \cref{section_construction_of_K_M}. These results strengthen analogous results in \cite{ITA_2008__42_1_105_0}. Ressayre's result shows that every RCEF which is a model of a certain theory (for instance, $\mathsf{ExpField + MaxVal(\exp)}$), ``produces'' a model of a certain arithmetic theory (for instance, $\mathsf{IOpen(\exp)}$). However, this does not allow one to obtain an $x^y$-integer part from a given RCEF.

To prove the converse of the theorems from \cref{section_ip_of_exp_fields}, we need to be able to build some exponential field containing the model $(\mathcal{M}^+, \exp)$ as an exponential integer part. But having only the usual unary exponentiation in $\mathcal{M}^+$, it is problematic to construct such a field. For example, we need to understand what value an expression of the form $\exp(\frac{1}{m})$ should take, where $m \in M^{>0}$. It must satisfy the property that for all $a, b\in {M}^{>0}$,
$$\frac{a}{b} < \exp\left(\frac{1}{m}\right) \iff \frac{a^m}{b^m} < 2,$$ 
but $a^m$ and $b^m$ are not defined. To do this, we consider the language with the power function $x^y$. With $(\mathcal{M}^+, x^y) \vDash \mathsf{IOpen} + \mathsf{T_{x^y}}$ (or, since $\mathsf{IOpen(x^y)} \vdash \mathsf{T_{x^y}}$, $(\mathcal{M}^+, x^y) \vDash \mathsf{IOpen(x^y)}$) we will be able to construct an exponential real closed field containing $(\mathcal{M}, x^y)$ as an $x^y$-integer part which satisfies the inequality $\exp(x) \geqslant 1 + x$ and thus to prove \autoref{theorem_(M,x^y)_is_iopen+T_x^y<=>ip_of_ExpField_RCF+Bern)}. Following \cite{Carl_Krapp_2021, krapp:2019}, we will denote this exponential field by $(\mathcal K_{\mathcal{M}}, \exp_{\mathcal{M}})$. The idea of the construction of $(\mathcal K_{\mathcal{M}}, \exp_{\mathcal{M}})$, which is presented in \cref{section_construction_of_K_M}, comes from the paper of M. Carl and L. Krapp (\cite[Section 2]{Carl_Krapp_2021} or \cite[Section 7.2]{krapp:2019}); however, in their setting $\mathcal{M}^+$ is a model of $\mathsf{PA}$. We slightly change his construction so that most of the proofs pass under the much weaker condition of $(\mathcal{M}^+, x^y)$ being a model of $\mathsf{IOpen} + \mathsf{T_{x^y}}$. Also, in \cite{Carl_Krapp_2021} the following interesting fact was proven: if $\mathcal{M}^+ \vDash \mathsf{PA}$ and the field $(\mathcal K_{\mathcal{M}}, \exp_{\mathcal{M}})$ is model complete, then it is o-minimal (\cite[Proposition 4.11]{Carl_Krapp_2021}). Moreover, if Schanuel's Conjecture holds, then under the same conditions we have $(K_{\mathcal{M}}, \exp_{\mathcal{M}}) \vDash Th_{\exists}(\R, \exp)$ (\cite[Corollary 4.12]{Carl_Krapp_2021}). Recall that Schanuel's Conjecture states that if $z_1, \dots, z_n \in \mathbb C$ ($\R$ in the real version) are linearly independent over $\mathbb{Q}$, then the transcendence degree of $\mathbb{Q}(z_1, \dots, z_n, e^{z_1}, \dots, e^{z_n})$ over $\mathbb{Q}$ is at least $n$. So far it has been proved that $(K_{\mathcal{M}}, \exp_{\mathcal{M}})$ is o-minimal (and also a model of $\Th(\R, \exp)$) only when $\mathcal{M}^+\vDash \Th(\mathbb{N})$ (\cite[Theorem 4.5]{Carl_Krapp_2021}).

It is worth mentioning that results of \cref{section_ip_of_exp_fields} heavily rely on the results on exponential fields and exponential equations, particularly, on those of L. van den Dries \cite{vandendries:1984} and A. Khovanskii \cite{khovankii}. The first paper concerns exponential fields in the most general setting and contains several useful results on exponential polynomials. In the second paper, it was proved that the set of roots of a system of exponential equations  (or, more generally, Pfaffian equations) in $\R$ has finitely many connected components. Moreover, the bound on the number of these components is recursive. It follows then that existential formulas in $\mathcal L_{\OR}(\exp)$ with one free variable define sets of a finite type (more precisely, a finite union of intervals and points). One important consequence of this result is the fact that the standard exponential field $(\R, \exp)$ is o-minimal, i.e., every definable subset of $\R$ is a finite union of intervals and points: by the famous result of A. Wilkie (see \cite[Second Main Theorem]{wilkie:1996}) $(\R, \exp)$ is model complete, hence, every formula is equivalent to an existential one, then apply the result of A. Khovanskii. The line of research of $(\R, \exp)$ is mostly motivated by an old open problem posed by A. Tarski whether the theory $\Th(\R, \exp)$ is decidable. Towards a solution of this problem, A. Wilkie and A. Macintyre have obtained the following result (based on the work \cite{wilkie:1996}): under the real version of Schanuel's Conjecture this theory is decidable ({\cite[Theorem 1.1]{Macintyre1996OnTD}}). Theories of other expansions of the ordered field $\R$ and models of thereof were studied extensively, see, for example, \cite{denef_dries_1988, MILLER1994,vandendries_macintyre_marker:1994,vandendries_macintyre_marker:1997,vandendries_macintyre_marker:2001} and others.

Models of arithmetical theories as exponential integer parts were studied by E. Jeřábek in \cite{jerabek2022}, \cite{jerabek2024} and by S. Boughattas and J.-P. Ressayre in \cite{ITA_2008__42_1_105_0}. E. Jeřábek~\cite{jerabek2022} shows that every countable model of a two-sorted arithmetical theory $\mathsf{VTC^0}$ is an exponential integer part of an RCEF. Since the theories $\mathsf{VTC^0}$ and $\mathsf{IOpen} + \mathsf{T_{x^y}}$ are incomparable (that is, neither theory implies the other), this result is incomparable with our result described above (\autoref{theorem_(M,x^y)_is_iopen+T_x^y<=>ip_of_ExpField_RCF+Bern)}). In~\cite{jerabek2024}, theories of exponential integer parts of RCEF in the languages with $\exp$, $P_2$ (the unary predicate for powers of 2) and in the pure language of ordered rings were axiomatized. As a consequence of the latter result E. Jeřábek established that not every model of $\mathsf{IOpen}$ has an elementary extension to an exponential integer part of an RCEF. However, the following question remains unanswered in \cite{jerabek2024}: whether \emph{every} model of the theory of exponential integer parts of RCEF in the language with $\exp$ can be embedded into a RCEF as an exponential integer part. In \cite{ITA_2008__42_1_105_0}  the following results were proved: If $\mathcal M^+$ is an $x^y$-integer part of a model of $\mathsf{ExpField} + \mathsf{IntVal}(\mathcal L_{\OR}(\exp)) + \mathsf{MaxVal}(\mathcal L_{\OR}(\exp)) + (x > 2 \to x^2 < \exp(x))$ (resp., $\Th(\R, \exp)$), then it is a model of $\mathsf{LOpen}(\exp)$ (resp., $\mathsf{LOpen(x^y)}$) (here $\mathsf{LOpen(\dots)}$ stands for the least element scheme in the corresponding language). We will strengthen these results (\autoref{theorem_M_exp_ip_ExpField+MaxVal}, \autoref{theorem_(M, x^y)_ip_R_exp} and \cref{prop_strengthening}) by replacing  $\mathsf{ExpField} + \mathsf{IntVal}(\mathcal L_{\OR}(\exp)) + \mathsf{MaxVal}(\mathcal L_{\OR}(\exp)) + (x > 2 \to x^2 < \exp(x))$ with $\mathsf{ExpField} + \mathsf{MaxVal}(\mathcal L_{\OR}(\exp)) + (\exp(1) = 2)$ and by replacing $\Th(\R, \exp)$ with a recursive subtheory of it.

In \cref{section_nonstandard_model} we construct nonstandard models of $\mathsf{IOpen}(\exp)$ and $\mathsf{IOpen(x^y)}$ using the o-minimal exponential field $\mathbb{R}((t))^{\LE}$ introduced in \cite{vandendries_macintyre_marker:2001} by L. van den Dries, A. Macintyre and D. Marker and theorems from \cref{section_ip_of_exp_fields}. Then, similarly to J. Shepherdson \cite{shepherdson:1964}, we obtain some independence results for these theories, for example, the irrationality of $\sqrt{2}$ and Fermat's Last Theorem for $n = 3$ are not provable. Let us mention that the theories $\mathsf{IOpen}(\exp)$ and $\mathsf{IOpen(x^y)}$ are sound (and even can be considered as subtheories of Peano arithmetic under a suitable interpretation of $\exp$ and $x^y$), so they cannot prove false statements such as the rationality of $\sqrt{2}$ or the negation of Fermat's Last Theorem. Hence, in order to prove the independence of these statements we only need to show their unprovability. Finally, we note that Shepherdson's model is recursive, so his result implies  that Tennenbaum's Theorem does not hold for $\mathsf{IOpen}$. A similar question for $\mathsf{IOpen}(\exp)$ and $\mathsf{IOpen(x^y)}$ seems to be open (our model is far from being recursive).

In \cref{section_further_results} we discuss some open questions and briefly mention some further results concerning extensions of $\mathsf{IOpen}$ in the language with $\exp$, which are under preparation.

\section{Preliminaries}\label{section_preliminaries}

\subsection{Conventions and notations}

By a ring we mean an associative commutative unitary ring. Usually, structures will be denoted by calligraphic letters (such as $\mathcal{M, F, R}, \dots$), and their domains will be denoted by $M, F, R, \dots$. Given a language $\mathcal L$, an $\mathcal L$-structure $\mathcal M$ and a function symbol $f \notin \mathcal L$ we denote by $\mathcal L(f)$ the expansion of $\mathcal L$ by $f$ and by $(\mathcal M, f_{\mathcal M})$ the expansion of $\mathcal M$ by function $f_{\mathcal M}$. If there is no confusion, we will omit the subscript $\mathcal M$ for interpretations of symbols from $\mathcal L$ in a structure $\mathcal M$. For example, if $\mathcal M = (M, +_\mathcal M, -_\mathcal M, \cdot_\mathcal M, 0_\mathcal M, 1_\mathcal M)$ is a ring, we may denote it by $(M, +, -, \cdot, 0, 1)$. We will denote the fact that $\mathcal M$ is a substructure of $\mathcal N$ as $\mathcal M \subseteq \mathcal N$ and the fact that $\mathcal M$ is an elementary substructure of $\mathcal N$ as $\mathcal M \preceq \mathcal N$. Given a structure $\mathcal M$ in a language $\mathcal L$ we will denote by $\Th(\mathcal M)$ the elementary theory of $\mathcal M$, i.e. the set of all $\mathcal L$-sentences which hold in $\mathcal M$. 

We will consider ordered semirings and rings (the precise definitions will be given later in this section) in the languages $$\mathcal L_{\OSR} := (+^{(2)}, \cdot^{(2)}, 0, 1, \leqslant^{(2)})$$ of ordered semirings and $$\mathcal{L}_{\OR} := (+^{(2)}, -^{(1)}, \cdot^{(2)}, 0, 1, \leqslant^{(2)})$$ of ordered rings respectively. 

If $\mathcal M = (M, \leqslant, \dots)$ is an ordered structure, then for $a \in M$ we set $M^{>a} := \{m \in M \mid  m > a\}$. For a subset $A \subseteq M$ and an element $a_0 \in M$ we write $A > a_0$ ($A \geqslant a_0$) if for all $a \in A$, $a > a_0$ (resp., $a \geqslant a_0$). Also we write $A < B$ if for all $a \in A$ and $b \in B$, $a < b$. If $\mathcal M = (M, +, -, \cdot, 0, 1, \leqslant)$ is an ordered ring, then we denote by $\mathcal M^+$ the semiring of nonnegative elements $(M^+, +, \cdot, 0, 1, \leqslant)$ with the domain $M^+ := \{a \in M \mid a \geqslant 0\}$. Also we will denote by $\R$ the ordered field of real numbers.

In order to simplify notation, we will omit $\forall$-quantifiers at the beginning of formulas when writing axioms of a theory. For instance, we will write 
$x + y = y + x$ instead of $\forall x \forall y (x + y = y + x)$. We will write $T_1 + T_2$ for deductive closure of the union of theories $T_1$ and $T_2$.

\subsection{Ordered rings and fields}

\begin{definition}
    $\mathsf{OR}$ is a theory in the language $\mathcal{L}_{\OR}$, consisting of the following axioms:
    \begin{enumerate}[start=0,label={(OR\arabic*)}]
        \item $x + (y + z) = (x + y) + z$;
        \item $x + y = y + x$;
        \item $x + 0 = x$;
        \item $x \cdot (y \cdot z) = (x \cdot y) \cdot z$;
        \item $x \cdot y = y \cdot x$;
        \item $x \cdot 1 = x$;
        \item $x \cdot (y + z) = x \cdot y + x \cdot z$;
        \item $x \leqslant x$;
        \item $(x \leqslant y \wedge y \leqslant x) \to x = y$;
        \item $(x \leqslant y \wedge y \leqslant z) \to z \leqslant z$;
        \item $x \leqslant y \vee y \leqslant x$;
        \item $x \leqslant y \to x + z \leqslant y + z$;
        \item $(x \leqslant y \wedge 0 \leqslant z) \to x \cdot z \leqslant y \cdot z$;
        \item $0 < 1$;
        \item $x + (-x) = 0$.
    \end{enumerate}
    Models of $\mathsf{OR}$ will be called \emph{ordered rings} (OR for short). 
\end{definition}

\begin{definition}
    $\mathsf{OF}$ is a theory in the language $\mathcal{L}_{\OR}$, consisting of the theory $\mathsf{OR}$ and the axiom
    $$(x \ne 0) \to \exists y (x \cdot y = 1).$$
    Models of $\mathsf{OF}$ will be called \emph{ordered fields} (OF for short). 
\end{definition}

\begin{remark}
    As usual, we will omit the symbol of multiplication $\cdot$ in the rest of the paper.
\end{remark}

\subsection{Discretely ordered rings and semirings}

\begin{definition}
    $\mathsf{DOR}$ is a theory in the language $\mathcal{L}_{\OR}$, consisting of the theory $\mathsf{OR}$ and the axiom
        $$x \leqslant 0 \vee 1 \leqslant x,$$
    which says that the order is discrete. Models of $\mathsf{DOR}$ will be called \emph{discretely ordered rings} (DOR for short). 
\end{definition}

\begin{definition}
    $\mathsf{DOSR}$ is a theory in the language $\mathcal{L}_{\OSR}$, consisting of the axioms (OR0)--(OR13) and the following axioms:
    $$x = 0 \vee 1 \leqslant x;$$
    $$x \leqslant y \leftrightarrow \exists z (x + z = y).$$
    Models of $\mathsf{DOSR}$ will be called \emph{discretely ordered semirings} (DOSR for short). 
\end{definition}

The theory $\mathsf{DOSR}$ will serve us as the basic arithmetic theory.

\begin{remark}
    One can freely move from models of $\mathsf{DOR}$ to $\mathsf{DOSR}$ and back. More precisely, for every DOR its nonnegative part is a DOSR and every DOSR is isomorphic to a nonnegative part of some DOSR (just consider the ring of pairs $(m, n)$ modulo an equivalence relation $(m, n) \sim (m', n') :\iff m + n' = m' + n$). In the rest of the paper we are mostly dealing with DORs and their nonnegative parts without any mentioning of this equivalence.  For more details see, for instance, \cite{kaye}.
\end{remark}

\begin{definition}
    $\mathsf{IOpen}$ is a theory in the language $\mathcal{L}_{\OSR}$, consisting of the theory $\mathsf{DOSR}$ and the open induction schema
    \begin{align*}
        \Big(\varphi(0, \overline{y}) \wedge \forall x \big(\varphi(x, \overline{y}) \rightarrow \varphi(x + 1, \overline{y})\big)\Big) \rightarrow \forall x \ \varphi(x, \overline{y}),
    \end{align*}
    where $\varphi(x, \overline{y})$ is a quantifier-free formula in the language $\mathcal{L}_{\OSR}$.
\end{definition}

\subsection{Real closed fields}

\begin{definition}
    Given a term $t$ in a language expanding $\mathcal L_{\OR}$, we denote by $\mathtt{IntVal}_{t}$ the formula
    $$(x < y \wedge t(x, \overline{w}) \leqslant 0 \wedge t(y, \overline{w}) \geqslant 0) \rightarrow \exists z ( (x \leqslant z \leqslant y) \wedge t(z, \overline{w}) = 0)$$
    and by $\mathtt{MaxVal}_{t}$ the formula 
    $$(x < y) \rightarrow \exists z_{\max}( x \leqslant z_{\max} \leqslant y \wedge \forall z ( x \leqslant z \leqslant y \to t(z, \overline{w}) \leqslant t(z_{\max}, \overline{w}))).$$
    Also we denote by $\mathsf{IntVal}(\mathcal L)$ the scheme of intermediate value theorems, i.e.
    $$\{\mathtt{IntVal}_{t}\mid  t \text{ is a term in the language }\mathcal{L}\},$$ 
    and by $\mathsf{MaxVal}(\mathcal L)$ the scheme of extreme value theorems, i.e.
    $$\{\mathtt{MaxVal}_{t}\mid  t \text{ is a term in the language }\mathcal{L}\}.$$ 
\end{definition}

\begin{definition}
    $\mathsf{RCF}$ is a theory in the language $\mathcal{L}_{\OR}$, consisting of the theory $\mathsf{OF}$ and the scheme $\mathsf{IntVal}(\mathcal L_{\OR})$. Models of $\mathsf{RCF}$ will be called \emph{real closed fields} (RCF for short).  
\end{definition}

Of course, the theory above is not a unique axiomatization of the class of real closed fields. The following results are well-known, more details can be found, for example, in \cite{lang} and \cite{marker}.

\begin{theorem}\label{theorem_equiv_defs_of_rcf}
    Let $\mathcal F$ be an ordered field. Then the following are equivalent:
    \begin{enumerate}[(1)]
        \item $\mathcal F$ is real closed;
        \item Every positive $a \in F$ has a square root in $\mathcal F$ and every polynomial of odd degree over $\mathcal F$ has a root;
        \item $\mathcal F \vDash \Th(\mathbb{R})$.
    \end{enumerate}
\end{theorem}

\begin{definition}
    If $\mathcal F \subseteq \mathcal R$ are ordered fields, then $\mathcal R$ is called a \emph{real closure} of $\mathcal F$ if $\mathcal R$ is real closed and the extension $\mathcal F \subseteq \mathcal R$ is algebraic.
\end{definition}

\begin{theorem}[{\cite[Chapter XI, Theorem 2.9]{lang}}]\label{theorem_real_closure_exists_and_unique}
    For every ordered field $\mathcal F$ there exists its real closure $\mathcal R$. Moreover, this $\mathcal R$ is unique up to an isomorphism fixing $F$.
\end{theorem}

\begin{remark}
    Due to the theorem above one can speak about \emph{the} real closure of a given ordered field.
\end{remark}

\begin{lemma}[{\cite[Chapter XI, Lemma 3.4]{lang}}]\label{theorem_relative_real_closure}
    If $\mathcal R$ is a real closed field and $\mathcal M \subseteq \mathcal R$ is an ordered subring, then the relative algebraic closure of $\mathcal M$ in $\mathcal R$ is the real closure of (the fraction field of) $\mathcal M$.
\end{lemma}

\subsection{Integer parts and Shepherdson's Theorem}

\begin{definition}
    Let $\mathcal M$ and $\mathcal R$ be ordered rings, $\mathcal M \subseteq \mathcal R$ and $\mathcal M$ is DOR. Then $\mathcal M$ is called an \textit{integer part of} $\mathcal R$ if for all $r \in R$ there is $m \in M$ such that $m \leqslant r < m + 1$. Notation: $\mathcal M \subseteq^{\IP} \mathcal R$.
\end{definition}

\begin{remark}
    Since $\mathcal M$ is discretely ordered, for every $r \in R$ its integer part is uniquely defined.
\end{remark}

Let $\mathcal M$ be a discretely ordered ring. Denote by $\mathcal M^+$ the semiring of the nonnegative elements of $\mathcal M$, by $\mathcal{F}(\mathcal M)$ the quotient field of $\mathcal M$ and by $\mathcal{R}(\mathcal M)$ the real closure of $\mathcal{F}(\mathcal M)$ (which is unique up to an isomorphism).

\begin{theorem}[J. Shepherdson, \cite{shepherdson:1964}]\label{shepherdson_theorem}
    Let $\mathcal{M}$ be a discretely ordered ring. Then $\mathcal{M}^+ \vDash \mathsf{IOpen}$ iff $\mathcal{M} \subseteq^{\IP} \mathcal{R}(\mathcal{M})$.
\end{theorem}

According to \autoref{theorem_real_closure_exists_and_unique}, \cref{theorem_relative_real_closure} and \autoref{theorem_equiv_defs_of_rcf}, we can reformulate \autoref{shepherdson_theorem} in the following form:

\begin{theorem}\label{shepherdson_theorem_reformulated}
    Let $\mathcal{M}$ be a discretely ordered ring. Then $\mathcal{M}^+ \vDash \mathsf{IOpen}$ iff there exists an ordered field $\mathcal{R}$ such that $\mathcal{M} \subseteq^{\IP} \mathcal{R}$ and $\mathcal{R} \vDash \Th(\mathbb{R})$.
\end{theorem}

\subsection{Extensions of $\mathsf{IOpen}$, exponential fields and exponential integer parts}

Now we define the theories $\mathsf{IOpen}(\exp)$ and $\mathsf{IOpen(x^y)}$, where $\exp$ is a new unary function symbol and $x^y$ is a new binary function symbol.

\begin{definition}
    $\mathsf{IOpen}(\exp)$ is a theory in the language $\mathcal{L}_{\OSR}(\exp)$, consisting of $\mathsf{DOSR}$, axioms for exponentiation
    \begin{enumerate}
        \item[(E1)] $\exp(0) = 1$,
        \item[(E2)] $\exp(x + 1) = \exp(x) + \exp(x)$
    \end{enumerate}
    and the induction scheme for quantifier-free formulas in the language $\mathcal L_{\OSR}(\exp)$.
\end{definition}

\begin{definition}
    $\mathsf{IOpen(x^y)}$ is a theory in the language $\mathcal{L}_{\OSR}(x^y)$, consisting of $\mathsf{DOSR}$, axioms for power function
    \begin{enumerate}
        \item[(P1)] $x^0 = 1$,
        \item[(P2)] $x^{y + 1} = x ^y \cdot x$
    \end{enumerate}
    and the induction scheme for quantifier-free formulas in the language $\mathcal L_{\OSR}(x^y)$.
\end{definition}

\begin{remark}
    $\exp$ stands for the base-2 exponentiation and $x^y$ stands for the power function.
\end{remark}

\begin{definition}
    $\mathsf{T_{x^y}}$ is a theory in the language $\mathcal{L}_{\OSR}(x^y)$, consisting of the following axioms ($1 + 1$ is denoted by $2$ for short):
\begin{enumerate}
    \item[(T1)] $x^0 = 1$,
    \item[(T2)] $x^1 = x$,
    \item[(T3)] $1^x = 1$,
    \item[(T4)] $x^{y + z} = x^y \cdot x^z$,
    \item[(T5)] $(x \cdot y)^z = x^z \cdot y^z$,
    \item[(T6)] $(x^y)^z = x^{yz}$,
    \item[(T7)] $(x > 1) \to (y < z \leftrightarrow x^y < x^z)$,
    \item[(T8)] $(x > 0) \to (y < z \leftrightarrow y^x < z^x)$,
    \item[(T9)] $(x > 0) \to \exists y (2^y \leqslant x < 2^{y + 1})$,
    \item[(T10)] $(y > 0) \to \exists z (z^y \leqslant x < (z + 1)^y)$,
    \item[(T11)] $(x > 0 \wedge y > 0) \to \Big( x^z y + z y^{z + 1} \geqslant y^{z + 1} + z xy^z \Big)$.
\end{enumerate}
\end{definition}

\begin{remark}[1]
    Informally saying, (T9) stands for the existence of the integer part of $\log x$ (base-2 logarithm) for $x > 0$, (T10) stands for the existence of the integer part of $\sqrt[y]{x}$ for $y > 0$. (T11) stands for the Bernoulli inequality, it can be informally written as $$(x > 0 \wedge y > 0) \to \Big( \big(\frac{x}{y}\big)^z \geqslant 1 + z\big(\frac{x}{y} - 1\big)\Big).$$
    We will give this remark a precise meaning in \cref{section_construction_of_K_M}.
\end{remark}

\begin{remark}[2]
    As we will see in \cref{section_ip_of_exp_fields}, $\mathsf{IOpen(x^y) \vdash T_{x^y}}$.
\end{remark}

Also, we will need the following definitions.

\begin{definition}
    $\mathsf{ExpField}$ is a theory in the language $\mathcal{L}_{\OR}(\exp)$, consisting of the theory $\mathsf{OF}$ and the following axioms:
    \begin{enumerate}[start=0,label={(EF\arabic*)}]
        \item $\exp(x) > 0$;
        \item $\exp(x + y) = \exp(x)\exp(y)$;
        \item $(x < y) \to (\exp(x) < \exp(y))$;
        \item $(x > 0) \to \exists y(x = \exp(y))$.
    \end{enumerate}
    Models of $\mathsf{ExpField}$ will be called \emph{exponential fields}.
\end{definition}

\begin{remark}
    It is not very hard to show that in fact (EF0) is provable in $\ExpField$ + (EF1)--(EF2).
\end{remark}

We will often denote exponential fields as $(\mathcal F, \exp)$, where $\mathcal F$ is an ordered field. Also we denote by $\log$ the function $\exp^{-1}\colon F^{>0} \to F$.

\begin{definition}\label{def_x^y}
    Given an exponential field $(\mathcal F, \exp)$ we define a function $x^y\colon F^+ \times F^+ \to F^+$ in the following (obvious) way:
    \begin{itemize}
        \item if $x > 0$, then $x^y := \exp(y \log x)$;
        \item if $y > 0$, then $0^y := 0$;
        \item $0^0 := 1$.
    \end{itemize}
\end{definition}

\begin{definition}\label{def_exp_ip}
    Let $(\mathcal F, \exp)$ be an exponential field, $\mathcal M \subseteq \mathcal F$ be an ordered ring. Then $\mathcal M$ is called an \textit{exponential integer part} of $(\mathcal F, \exp)$ if $\mathcal M \subseteq^{\IP} \mathcal F$ and $M^+$ is closed under $\exp$ (i.e. for all $m \in M^+$ we have $\exp(m) \in M^+$). Notation: $\mathcal M \subseteq^{\IP}_{\exp} (\mathcal F, \exp)$.
\end{definition}

\begin{definition}\label{def_x^y_ip}
    Let $(\mathcal F, \exp)$ be an exponential field, $\mathcal M \subseteq \mathcal F$ be an ordered ring. Then $\mathcal M$ is called an \textit{$x^y$-integer part} of $(\mathcal F, \exp)$ if $\mathcal M \subseteq^{\IP} \mathcal F$ and $M^+$ is closed under $x^y$ (i.e. for all $n, m \in M^+$ we have $m^n \in M^+$). Notation: $\mathcal M \subseteq^{\IP}_{x^y} (\mathcal F, \exp)$.
\end{definition}

\begin{remark}
    Note that in \cref{def_x^y_ip}, the base of exponentiation $\exp(1)$ does not have to be an ``integer''. This is in contrast to \cref{def_exp_ip}, where the base $\exp(1)$ is required to lie in $M^+$. Typically, in the context of \cref{def_exp_ip}, we have $\exp(1) = 2$.
\end{remark}

\subsection{Khovanskii's Theorem and O-minimal structures}

We will need one important result by A. Khovanskii (see \cite{khovankii}). To formulate it we need the following definition.

\begin{definition}
    A sequence of differentiable functions $(f_1, \dots, f_k)$, $f_i\colon \R^{n} \to \R$, is called a \emph{Pfaffian chain} if 
    for all $i = 1, \dots, k$ and $j = 1, \dots, n$ there is a polynomial $p_{i, j} \in \R[X_1, \dots, X_n, Y_1, \dots, Y_i]$ such that for all $\overline x = (x_1, \dots, x_n) \in \R^n$ we have
    $$\frac{\partial f_i}{\partial x_j}(\overline x) = p_{i, j}(\overline x, f_1(\overline x), \dots, f_i(\overline x)).$$
    
    A \emph{Pfaffian equation} is an equation of the form 
    $$p(\overline x, f_1(\overline x), \dots, f_k(\overline x)) = 0,$$
    where $(f_1, \dots, f_k)$ is a Pfaffian chain and $p \in \R[X_1, \dots, X_n, Y_1, \dots, Y_k]$. \emph{Complexity} of a given Pfaffian equation is a sequence of the following numbers: $n$, $k$, $(\deg p_{i, j})_{i, j}$ and $\deg p$.
\end{definition}

\begin{theorem}[{\cite[Theorem 4]{khovankii}}]\label{theorem_khovanskii}
    Given a Pfaffian equation $p(\overline x, f_1(\overline x), \dots, f_k(\overline x)) = 0$ there is a number $N \in \N$ such that the set of solutions 
    $$\{\overline x \in \R^n \mid  p(\overline x, f_1(\overline x), \dots, f_k(\overline x)) = 0\}$$
    has no more than $N$ connected components. Moreover, $N$ can be found effectively from the complexity of the equation.
\end{theorem}

\begin{corollary}\label{corollary_khovaskii}
    Given a Pfaffian equation $p(x, \overline y, f_1(x, \overline y), \dots, f_k(x, \overline y)) = 0$ the set 
    $$\{x \in \R \mid \exists \overline y \in \R^n\colon p(x, \overline y, f_1(x, \overline y), \dots, f_k(x, \overline y)) = 0\}$$
    is a union of no more than $N$ intervals and points, where $N$ is from the theorem above for the equation 
    $$p(x, \overline y, f_1(x, \overline y), \dots, f_k(x, \overline y)) = 0.$$
\end{corollary}

Now suppose that $p(x, \overline y, f_1(x, \overline y), \dots, f_k(x, \overline y))$ is expressible by a term $t(x, \overline y, \overline a)$ with parameters $\overline a$ from $\R$ in some expansion of the field of real numbers (for instance, in $\mathbb{R}_{\exp} := (\R, +, -, \cdot, 0, 1, \leqslant, e^x)$). Then \cref{corollary_khovaskii} can be expressed in the language of this expansion by the following formula:
$$\bigvee\limits_{N' = 0}^N \bigvee\limits_{c = 0}^{N'} \exists z_1 \dots \exists z_c \exists r_1 \dots \exists r_{N' - c} \exists l_1 \dots \exists l_{N' - c} \forall x (\exists \overline y(t(x, \overline y, \overline a) = 0) \leftrightarrow (\bigvee\limits_{i = 1}^c (x = z_i) \vee \bigvee\limits_{i = 1}^{N' - c} (l_i < x < r_i))),$$
where $N'$ stands for the total number of intervals and points, $c$ and $N' - c$ stand for the number of points and the number of intervals respectively, $z_1, \dots, z_c$ are these points and $(l_1, r_1), \dots, (l_{N' - c}, r_{N' - c})$ are these intervals. Let us denote the universal closure of this formula as $\mathtt{KTB}_t^N$ (KTB means ``Khovanskii's Theorem bound''). So, we have that $\mathtt{KTB}_t^N$ holds in the expansion under consideration.

Next, it is easy to see that for all $\mathcal L_{\OR}(\exp)$-terms with parameters $t(x, \overline y, \overline a)$ the equation $t(x, \overline y, \overline a) = 0$ is a Pfaffian equation, it can be shown by induction on $t$. Hence, $\mathtt{KTB}_t^N \in \Th(\R_{\exp})$. This motivates the following definition.

\begin{definition}
    $\mathsf{KTB}$ is a scheme in the language $\mathcal L_{\OR}(\exp)$ of the sentences $\mathtt{KTB}_t^N$ for all $\mathcal L_{\OR}(\exp)$-terms $t$ and $N$ from \cref{corollary_khovaskii}.
\end{definition}

\begin{remark}
    As we have already shown, $\mathsf{KTB} \subseteq \Th(\R_{\exp})$. Moreover, $\mathsf{KTB}$ is recursive. Next, it is not very hard to see that $\mathsf{ExpField + KTB} \vdash \mathsf{RCF}$, since in every model $(\mathcal R, \exp) \vDash \mathsf{ExpField + KTB}$, for every polynomial $p(X) \in \mathcal R[X]$, the set $\{x \in R\mid  p(x) > 0\}$ is a finite union of intervals and points (this easily implies $\mathsf{IntVal}(\mathcal L_{\OR})$). 
\end{remark}

\section{Integer parts of exponential fields}\label{section_ip_of_exp_fields}

In this section we obtain some sufficient conditions for discretely ordered semirings with exponentiation or power function to be a model of a certain extension of $\mathsf{IOpen}$.

\begin{theorem}\label{theorem_M_exp_ip_ExpField+MaxVal}
    Let an exponential field $(\mathcal{R}, \exp)$ be a model of $\mathsf{ExpField} + \mathsf{MaxVal}(\mathcal L_{\OR}(\exp)) + \exp(1) = 2$ and $\mathcal{M} \subseteq^{\IP}_{\exp} (\mathcal{R}, \exp)$. Then $(\mathcal{M}^+, \exp) \vDash \mathsf{IOpen}(\exp)$.
\end{theorem}

Before the proof let us cite the following lemmas by L. van den Dries. Here $(\mathcal F, E)$ is an arbitrary exponential field with exponentiation $E$ and by $\mathcal F[X]^E$ we denote the ring of exponential polynomials over $(\mathcal F, E)$. The definition of the latter can be found in \cite[1.1]{vandendries:1984}. Informally saying, $\mathcal F[X]^E$ is a structure that contains the polynomial ring $\mathcal F[X]$ and is closed under $E$. By $\mathcal F[x]^E$ we denote the ring of exponential functions, i.e., the least set of functions that contains $F$, $\id_F$, and is closed under $+$, $\cdot$ and $E$. For an exponential polynomial $p \in \mathcal F[X]^E$ one can define the exponential function $\hat p \in \mathcal F[x]^E$ such that $\hat p(x)$ equals the value of the exponential polynomial $p$ at $x$.

\begin{lemma}[{\cite[Lemma 3.2]{vandendries:1984}}]\label{lemma_derivation_in_exp_field}
    For all $r \in F$ there exists a unique formal\footnote{that is, an additive operator satisfying Leibniz's law $(p q)' = p q' + p' q$} derivative $'\colon \mathcal F[X]^E \to \mathcal F[X]^E$ such that $'$ is trivial on $F$, $X' = 1$ and $E(p)' = r \cdot p' \cdot E(p)$.
\end{lemma}

\begin{lemma}[{\cite[Lemma 3.3]{vandendries:1984}}]\label{lemma_ord_in_exp_field}
    There exists a map $\ord \colon \mathcal F[X]^E \to \Ord$ such that 
    \begin{enumerate}[(i)]
        \item $\ord(p) = 0$ iff $p = 0$;
        \item for all nonzero $p \in \mathcal F[X]^E$ either $\ord(p') < \ord(p)$ or there exists $q \in \mathcal F[X]^E$ such that $\ord(p \cdot E(q)) < \ord(p)$.
    \end{enumerate}
    Here $p'$ denotes the formal derivative from \cref{lemma_derivation_in_exp_field} for some $r$ and $\Ord$ denotes the class of ordinals.
\end{lemma}

\begin{lemma}[{\cite[Proposition 3.4]{vandendries:1984}}]\label{lemma_p'=0}
    For all $p \in \mathcal F[X]^E$ we have $p' = 0$ iff $p \in F$, where $p'$ denotes the formal derivative from \cref{lemma_derivation_in_exp_field} for some $r \ne 0$.
\end{lemma}

\begin{lemma}[{\cite[Proposition 4.1]{vandendries:1984}}]\label{lemma_p->hat_p_is_isomorphism}
    The map $p \mapsto \hat p$ is an isomorphism between $\mathcal F[X]^E$ and $\mathcal F[x]^E$.
\end{lemma}

\begin{remark}
    By \cref{lemma_p->hat_p_is_isomorphism} we can identify an exponential polynomial $p$ with an exponential function $\hat p$.
\end{remark}

\begin{lemma}[{\cite[Corollary 4.11]{vandendries:1984}}]\label{lemma_exp_poly_has_finite_num_of_roots}
    Let $(\mathcal F, E) \vDash \mathsf{MaxVal}(\mathcal L_{\OR}(\exp))$ and $(\mathcal F, E) \vDash \forall x (\exp(x) \geqslant 1 + x)$. Then every nonzero exponential polynomial has a finite number of roots.
\end{lemma}

Using these results one can prove the following.

\begin{lemma}\label{lemma_about_models_of_MaxVal(exp)}
    Let $(\mathcal R, \exp)$ be a model of $\mathsf{ExpField} + \mathsf{MaxVal}(\mathcal L_{\OR}(\exp)) + \exp(1) = 2$. Then 
    \begin{enumerate}[(i)]
        \item there exists $a \in R^{>0}$ such that $(\mathcal R, \exp) \vDash \forall x \ (\exp(a x) \geqslant 1 + x)$;
        \item $\exp$ is differentiable with $\exp' = a^{-1} \exp$ and for every exponential polynomial $p$ we have $\widehat{p'} = (\hat p)'$ with $r = a^{-1}$ (here $a$ is from (i));
        \item every nonzero exponential polynomial has a finite number of roots;
        \item if $t$ is an $\mathcal{L}_{\OR}(\exp)$-term and $(\mathcal R, \exp) \vDash \mathtt{IntVal}_{t}$, then, for all $\overline{a} \in R$, the set 
        $$\{x \in R\mid  (\mathcal R, \exp) \vDash t(x, \overline a) \leqslant 0\}$$ is a finite union of intervals and points; 
        \item $(\mathcal R, \exp) \vDash \mathsf{IntVal}(\mathcal L_{\OR}(\exp))$.
    \end{enumerate}
\end{lemma}

\begin{remark}[1]\label{remark_about_continuity}
    Continuity and derivative are defined in terms of  $\varepsilon$-$\delta$. That is, $f\colon D \to R$, $D \subseteq R^k$, $k \in \mathbb N$, is called continuous at the point $\overline x_0 = (x_{0, 1}, \dots, x_{0, k}) \in D$ if 
        $$\forall \varepsilon \in R^{>0} \exists \delta \in R^{>0} \forall \overline x \in R^k \left(\bigwedge\limits_{i = 1}^{k} |x_i - x_{0, i}| < \delta \implies |f(\overline x) - f(\overline x_0)| < \varepsilon\right)$$ 
    and for the function $f\colon R \to R$ we say that $f'(x_0) = b$ if
        $$\forall \varepsilon \in R^{>0} \exists \delta \in R^{>0} \forall x \in R \left(0 < |x - x_0| < \delta \implies \left|\frac{f(x) - f(x_0)}{x - x_0} - b\right| < \varepsilon\right).$$
    As in the standard case, $+$ and $\cdot$ are continuous and differentiable, the composition of continuous functions is continuous and the usual identities for the derivative hold, for instance, $(f \cdot g)'(x_0) = (f' \cdot g + f \cdot g')(x_0)$ and $(f \circ g)'(x_0) = (g' \cdot (f' \circ g))(x_0)$. Also, we will use the following property: if $f'(x_0) > 0$, then there exists $\varepsilon > 0$ such that, for all $x \in R$,we have $x_0 - \varepsilon < x < x_0 \implies f(x) < f(x_0)$ and $x_0 < x < x_0 + \varepsilon \implies f(x) < f(x_0)$.
\end{remark}

\begin{remark}[2]
    It follows from (ii) that $\exp$ is continuous.
\end{remark}

\begin{remark}[3]
    An analogous result to \cref{lemma_about_models_of_MaxVal(exp)}(v) for real closed fields is obtained in \cite[Theorem 2.2]{gamboa1987} (namely, $\mathsf{OF + MaxVal}(\mathcal L_{\OR})$ implies $\mathsf{IntVal}(\mathcal L_{\OR})$). The proof in \cite{gamboa1987} uses a related but distinct technique.
\end{remark}

\begin{proof}[Proof of \cref{lemma_about_models_of_MaxVal(exp)}]

    \begin{enumerate}
        \item [(i)] First we prove that $(\mathcal R, \exp) \vDash \forall x \geqslant 0 (\exp(x) \geqslant x)$. Suppose it is not the case. Then there is $x_0 \in R^+$ such that $\exp(x_0) < x_0$. By $\mathsf{MaxVal}(\mathcal L_{\OR}(\exp))$ the term $x - \exp(x)$ reaches the maximum on $[0, x_0]$ at some point $x^* \in [0, x_0]$. By our hypothesis, we have $x^* - \exp(x^*) > 0$. If there is an $n_0 \in \mathbb N$ such that $x^* < n_0$, then for some $n \in \mathbb N$ we have $n \leqslant x^* < n + 1$, so 
        $$x^* - \exp(x^*) \leqslant n + 1 - \exp(n) = n + 1 - 2^n \leqslant 0.$$ This implies that $x^*$ is infinite and $x^* - 1$ lies in the segment $[0, x_0]$. But $$(x^* -1 - \exp(x^* - 1)) - (x^* - \exp(x^*)) = -1 + \frac{\exp(x^*)}{2} > 0,$$ a contradiction with the choice of $x^*$.

        Now we prove that there exists $a \in R$ such that $(\mathcal R, \exp) \vDash \forall x (\exp (a x) \geqslant 1 + x)$. By $\mathsf{MaxVal(\mathcal L_{\OR}(\exp))}$ there exists $b \in R$ such that $(1 + b)\exp(-b)$ is the maximum value of the term $(1 + x) \exp(-x)$ on the segment $[-1, 10]$. Note that $(1 + b)\exp(-b) \geqslant (1 + 0)\exp(-0) = 1$. We want to prove that it is a global maximum of $(1 + x) \exp(-x)$. Indeed, for $x < -1$, the value of $(1 + x) \exp(-x)$ is negative. For $x > 10$, 
        $$\frac{1 + x}{\exp(x)} = \frac{6 + (x - 5)}{\exp(x - 5) 2^5} \leqslant \frac{6 + \exp(x - 5)}{\exp(x - 5) 2^5} < \frac{6}{2^{5}} + \frac{1}{2^5} < 1.$$ So, $b$ is the global maximum.

        Let $a := b + 1$. Clearly, $a > 0$ ($b \ne -1$ since $(1 - 1) \exp(1) = 0 < 1$). Then, replacing $x$ by $a(x + 1) - 1$, we have $$\frac{a(x + 1)}{\exp(a(x + 1) - 1)} \leqslant \frac{1 + b}{\exp(b)} = \frac{a}{\exp(a - 1)}$$ for all $x \in R$. Hence, $$\frac{a x + a}{\exp(a x + a - 1)} \leqslant \frac{a}{\exp(a - 1)}$$ and $\exp(a x) \geqslant 1 + x$.

        \item [(ii)] 

        Let us fix $0 < x < a$. Note that the following inequalities holds:
        $$
        \frac{\exp(x) - 1}{x} - \frac{1}{a} \geqslant \frac{x/a}{x} - \frac{1}{a} = 0
        $$
        (by(i)), and
        \begin{align*}
        \frac{\exp(x) - 1}{x} - \frac{1}{a}
        &= \frac{1/\exp(-x) - 1}{x} - \frac{1}{a} \\
        &\leqslant \frac{\frac{1}{1 - x/a} - 1}{x} - \frac{1}{a} \\
        &= \frac{\frac{a}{a - x} - 1}{x} - \frac{1}{a} \\
        &= \frac{a - (a - x)}{x(a - x)} - \frac{1}{a} \\
        &= \frac{1}{a - x} - \frac{1}{a} \\
        &= \frac{a - a + x}{a(a - x)} = \frac{x}{a(a - x)}.
        \end{align*}
        Similarly, for $-a < x < 0$ we obtain
        $$
        \frac{x}{a(a - x)}\leqslant \frac{\exp(x) - 1}{x} - \frac{1}{a} \leqslant 0
        $$
        and, hence, 
        $$
        \exp'(0) = a^{-1}.
        $$
        This implies that
        $$
        \exp' = \exp'(0) \cdot \exp = a^{-1} \exp.
        $$

        Alternatively, for item (ii) one can argue as in \cite[Theorem 14]{Dahn1983}. The rest of the statement can be easily proven by induction on the construction of $\mathcal R[X]^{\exp}$.

        \item [(iii)] Let $E(x) := \exp(ax)$. Clearly, $E$ is an exponentiation. Consider an exponential polynomial $p \in \mathcal R[X]^{\exp}$. Denote by $\tilde p \in \mathcal R[X]^{E}$ the exponential polynomial obtained from $p$ by replacing all occurrences of $\exp(q)$ by $E(a^{-1}q)$ (formally, $\tilde p$ is defined by induction). It is easy to see that $$\hat{\tilde{p}} = \hat p.$$ Now the desired result follows from the application of \cref{lemma_exp_poly_has_finite_num_of_roots} to the exponential field $(\mathcal R, E)$.

        \item [(iv)] For any tuple of parameters $(a_1, \dots, a_l) = \overline a \in R^l$ and an $\mathcal{L}_{\OR}(\exp)$-term $t(x, \overline a)$ we write $\mathbf t_{\overline a}$ for the exponential function $u \mapsto t(u, \overline a)$. By \cref{lemma_p->hat_p_is_isomorphism} we can identify $\mathbf t_{\overline a}$ with an exponential polynomial.
        
        Now fix $\overline a \in R$ and an $\mathcal L_{\OR}(\exp)$-term $t(x, \vec a)$. The case of $\mathbf t_{\overline a}$ equals zero is trivial, assume it is not the case. By (iii), $\mathbf t_{\overline a}$ has a finite number of roots, say, 
        $$u_1 < u_2 < ... < u_k.$$ Given that $(\mathcal R, \exp) \vDash \mathtt{IntVal}_t$, the function $\mathbf t_{\overline a}$ does not change sign on each interval of the form $$(-\infty, u_1), (u_1, u_2), \dots, (u_k, +\infty)$$ (otherwise, there will be $(k + 1)$-th root by $\mathtt{IntVal}_t$). So, $$\{u \in R\mid  (\mathcal R, \exp) \vDash t(u, \overline a) < 0\}$$ is a finite union of intervals and $$\{u \in R\mid  (\mathcal R, \exp) \vDash t(u, \overline a) \leqslant 0\} = \{u \in R\mid  (\mathcal R, \exp) \vDash t(u, \overline a) < 0\} \cup \{u_1, \dots, u_k\}$$ is a finite union of intervals and points. 

        \item [(v)] We proceed by induction on $\ord(\mathbf t_{\overline a})$ for an $\mathcal{L}_{\OR}(\exp)$-term $t(x, \overline a)$.

        If $\ord(\mathbf t_{\overline a}) = 0$, then $\mathbf t_{\overline a}(u) = 0$ for all $u \in R$ and $\mathtt{IntVal}_t$ holds.

        Let $\ord(\mathbf t_{\overline a}) > 0$. By \cref{lemma_ord_in_exp_field}, either there exists $q \in \mathcal R[X]^{\exp}$ such that $\ord(\mathbf t_{\overline a} \cdot \exp(q)) < \ord(\mathbf t_{\overline a})$ or $\ord(\mathbf t_{\overline a}') < \ord(\mathbf t_{\overline a})$. Consider the first case. We can choose some $\overline b \in R$ and an $\mathcal L_{\OR}$-term $s(x, \overline b)$ such that $\mathbf s_{\overline b} = q$ (by the construction of $\mathcal R[X]^{\exp}$). By the induction hypothesis, the intermediate value theorem holds for $t(x, \overline a) \cdot \exp(s(x, \overline b))$. Since $\exp(s(u, \overline b))$ is positive for all $u \in R$, the intermediate value theorem holds for $t(x, \overline a)$ as well.

        Now consider the second case, which is more complicated. Suppose, $\mathbf t_{\overline a} (l) < 0,\ \mathbf t_{\overline a} (r) > 0$ and there is no $u$ between $l$ and $r$ such that $\mathbf t_{\overline a}(u) = 0$. Clearly, there is an $\mathcal L_{\OR}(\exp)$-term $s(x, \overline a, a^{-1})$ such that $\mathbf s_{\overline a, a^{-1}} = \mathbf t_{\overline a}'$, where $a$ is from (i) (it can be obtained by induction on $t$). By the induction hypothesis, we have $(\mathcal R, \exp) \vDash \mathtt{IntVal}_s$. If $s_{\overline a, a^{-1}} = 0$, then by \cref{lemma_p'=0} we have that $\mathbf t_{\overline a}$ is a constant, a contradiction. So, by (iii), the set 
        $$X := \{l , r\} \cup \{u \in R\mid  (\mathcal R, \exp) \vDash s(u, \overline a, a^{-1}) = 0 \wedge l \leqslant u \leqslant r\}$$
        is finite, say, $X = \{u_0, u_1, \dots, u_n\}$, where $l = u_0 < u_1 < \dots < u_n = r$. Let $i > 0$ be the least natural number such that $\mathbf t_{\overline a}(u_i) > 0$ (such exists since $\mathbf t_{\overline a}(u_n) > 0$). Choose $l', r' \in R$ such that $u_{i - 1} < l' < r' < u_i$, $\mathbf t_{\overline a}(l') < 0$ and $\mathbf t_{\overline a}(r') > 0$ (by continuity of exponential polynomials such elements exist). Note that $\mathbf s_{\overline a, a^{-1}}$ does not change sign on the segment $[l', r']$ since there are no roots on it and we have $\mathtt{IntVal}_s$. W.l.o.g. we may assume that $\mathbf s_{\overline a, a^{-1}}(u) > 0$ for $l' \leqslant u \leqslant r'$.
        
        Denote by $u^*$ an element between $l'$ and $r'$ in which $\mathbf t^2_{\overline a}(x)$ reaches the maximum on $[l', r']$ and by $u_*$ an element between $l'$ and $r'$ in which $-\mathbf t^2_{\overline a}(x)$ reaches the maximum on $[l', r']$ (i.e. $\mathbf t^2_{\overline a}(x)$ reaches the minimum  on $[l', r']$). Such $u^*$ and $u_*$ exist by $\mathsf{MaxVal}(\mathcal L_{\OR}(\exp))$. Suppose $l' < u^* < r'$. If $\mathbf t_{\overline a}(u^*) > 0$, then there is an $u''$ such that $u^* < u'' < r'$ and $\mathbf t_{\overline a}(u'') > \mathbf t_{\overline a}(u^*) > 0$ (since $\mathbf t_{\overline a}'(u^*) > 0$). It is a contradiction, since $\mathbf t^2_{\overline a}(u'') > \mathbf t^2_{\overline a}(u^*)$. If $\mathbf t_{\overline a}(u^*) < 0$, then there is such $u''$ that $l' < u'' < u^*$ and $\mathbf t_{\overline a}(u'') < \mathbf t_{\overline a}(u^*) < 0$ (since $\mathbf t_{\overline a}'(u^*) > 0$). It is also a contradiction. So, $u^* \in \{l', r'\}$. In a  similar way one can obtain that $u_* \in \{l', r'\}$. If $u^* = u_*$, then $\mathbf t_{\overline a}^2$ is a constant, hence $0 = (\mathbf t_{\overline a}^2)' = 2 \mathbf t_{\overline a} \mathbf t_{\overline a}'$, so, $\mathbf t_{\overline a}' = 0$ (since $\mathbf t_{\overline a}(u) \ne 0$ for all $u \in [l', r']$). So, by \cref{lemma_p'=0}, $\mathbf t_{\overline a}$ is a constant, a contradiction. So, we have either $u^* = l', u_* = r'$ or $u^* = r', u_* = l'$. 

        Consider the first case. Then $\mathbf t_{\overline a}(u_*) > 0$. There is a $v$ such that $l' < v < u_* = r'$ and $\mathbf t_{\overline a} (u_*) > \mathbf t_{\overline a} (v) > 0$ (since $\mathbf t_{\overline a}'(u_*) > 0$ and $\mathbf t_{\overline a}$ is continuous). It is again a contradiction. The second case can be treated similarly.
    \end{enumerate}
\end{proof}

Now we are ready to prove \cref{theorem_M_exp_ip_ExpField+MaxVal}.

\begin{proof}[Proof of \cref{theorem_M_exp_ip_ExpField+MaxVal}]
    Consider a discretely ordered ring $\mathcal{M}$ and an exponential field $(\mathcal{R}, \exp)$ as in the statement of the theorem. Let $\varphi(x, \overline y)$ be a quantifier-free formula in the language $\mathcal{L}_{\OSR}(\exp)$. Fix a tuple of parameters $\overline a \in M^+$ and suppose $(\mathcal M^+, \exp) \vDash \varphi(0, \overline a) \wedge \exists x \neg \varphi(x, \overline a)$.

    Consider some terms $t_1(x, \overline a)$ and $t_2(x, \overline a)$ in the language $\mathcal L_{\OSR} (\exp)$. By \cref{lemma_about_models_of_MaxVal(exp)}(v) we have the intermediate value theorem for $t_1(x, \overline a) - t_2(x, \overline a)$. By \cref{lemma_about_models_of_MaxVal(exp)}(iv), we have that the set 
    $$\{x \in R\mid  (\mathcal R, \exp) \vDash t_1(x, \overline a) \leqslant t_2(x, \overline a)\}$$ 
    is a finite union of intervals and points. Then, the set 
    $$X_{\varphi}(\overline a) := \{x \in R\mid  (\mathcal R, \exp) \vDash \neg \varphi(x, \overline a) \wedge (x > 0)\}$$ 
    is a finite union of intervals and points (since it is a boolean combination of such sets). Also 
    $$\varnothing \neq \{x \in M^+ \mid (\mathcal M^+, \exp) \vDash \neg \varphi(x, \overline{a})\} \subseteq X_{\varphi}(\overline{a}).$$ Now choose the leftmost interval or the leftmost point in $X_\varphi(\overline{a})$ containing elements from $\{x \in M^+ \mid (\mathcal M^+, \exp) \vDash \neg \varphi(x, \overline{a})\}$. Consider two cases.
    \begin{enumerate}
        \item[(i)] Chosen a point $c$. Then $c \in M^{>0}$ (since $(\mathcal{M}^+, \exp) \vDash \varphi(0, \overline{a})$), so $c - 1 \in M^+$  and $(\mathcal{M}^+, \exp) \vDash \varphi(c - 1, \overline{a})$. Hence, $(\mathcal{M}^+, \exp) \vDash \neg\forall x(\varphi(x,\overline{a})\rightarrow\varphi(x + 1, \overline{a}))$, and the induction axiom holds for the formula $\varphi$. 
        \item[(ii)] Chosen an interval $(a, b)$. Let $m \in (a,b) \cap M^+$. Denote by $m'$ the integer part of $a$, i.e. such an element of $M^+$ that $m' \leqslant a < m' + 1$. Since $a < m$, then $m' + 1 \leqslant m < b$, so $m' + 1 \in (a, b)$. Then $(\mathcal{M}^+, \exp) \vDash \varphi(m', \overline{a}) \wedge \neg\varphi(m' + 1, \overline{a})$. Hence, the induction axiom holds for the formula $\varphi$.
    \end{enumerate}
    In both cases the induction axiom holds, so $(\mathcal M^+, \exp) \vDash \mathsf{IOpen(\exp)}$.
\end{proof}

\begin{theorem}\label{theorem_(M, x^y)_ip_R_exp}
    Let an exponential field $ (\mathcal{R}, \exp)$ be a model $\mathsf{ExpField} + \mathsf{KTB}$ and $\mathcal{M} \subseteq^{\IP}_{x^y} (\mathcal{R}, \exp)$. Then $(\mathcal{M}^+, x^y) \vDash \mathsf{IOpen(x^y)}$. 
\end{theorem}

\begin{proof}
    Let us fix parameters ($a_1, \dots, a_n) = \overline a \in M^+$, a quantifier-free formula $\varphi$ in a language $\mathcal L_{\OSR}(x^y)$ and suppose $(\mathcal M^+, x^y) \vDash \varphi(0, \overline a) \wedge \exists x \neg \varphi(x, \overline a)$. Note that formulas of the form $\neg(t_1 = t_2)$ and $\neg (t_1 \leqslant t_2)$ are equivalent to $(t_1 + 1 \leqslant t_2 \vee t_2 + 1 \leqslant t_1)$ and $t_2 + 1 \leqslant t_1$ respectively modulo $\mathsf{DOSR}$. So, one can eliminate all occurrences of $\neg$ and $\to$ in $\varphi$. Since we do not have a symbol of power function in the language $\mathcal L_{\OR}(\exp)$, we need to replace all occurrences of $x^y$ in the formula $\varphi$. We construct an $\mathcal L_{\OR}(\exp)$-formula $\varphi^*$ in such a way that the following holds: 
    $$\forall x \in M^{>0} \Big ( (\mathcal{M}^+, x^y) \vDash \varphi(x, \overline{a})\ \iff \ (\mathcal{R}, \exp) \vDash \varphi^*(x, \overline{a})\Big).$$
    The formula $\varphi^*$ is defined recursively as follows.
    \begin{enumerate}
        \item[(1)] If $\varphi = (v = v')$, where $v$, $v'$ are variables or constants, then $\varphi^* := (v = v')$.
        
        \item[(2)] If $\varphi = (v = t_1 + t_2)$, where $v$ is a variable or constant, $t_1$, $t_2$ are terms, then 
        $$\varphi^* := \exists z_1 \exists z_2 ((z_1 = t_1)^* \wedge (z_2 = t_2)^* \wedge  v = z_1 + z_2).$$
        
        \item[(3)] If $\varphi = (v = t_1 \cdot t_2)$, where $v$ is a variable or constant, $t_1$, $t_2$ are terms, then 
        $$\varphi^* := \exists z_1 \exists z_2 ((z_1 = t_1)^* \wedge (z_2 = t_2)^* \wedge v = z_1 \cdot z_2).$$
        
        \item[(4)] Let $\varphi = (v = t_1^{t_2})$, where $v$ is a variable or constant, $t_1$, $t_2$ are terms. Then define $\varphi^*$ as 
        \begin{align*}
            & \big( (0 = t_1)^* \wedge (0 = t_2)^* \wedge (v = 1) \big) \vee\\
            & \big( (0 = t_1)^* \wedge (\exists z_1\exists z_2 ((z_1 = t_2)^* \wedge z_1 z_2 = 1)) \wedge (v = 0) \big) \vee \\
            \exists z_1\exists z_2\exists z_3 & \big((z_1 = t_1)^* \wedge (z_2 = t_2)^* \wedge (\exp(z_3) = z_1)  \wedge v = \exp(z_2 z_3)\big).
        \end{align*}
        The first disjunct of $\varphi^*$ corresponds to the case where both $t_1$ and $t_2$ are zero; the second corresponds to $t_1 = 0$ and $t_2 \ne 0$; and the third covers the case where $t_1 \ne 0$.
        
        \item[(5)] If $\varphi = (t_1 + t_2 = t)$, then 
        $$\varphi^* := \exists z_1 \exists z_2 \exists z_3 ((z_1 = t_1)^* \wedge (z_2 = t_2)^* \wedge (z_3 = t)^* \wedge z_3 = z_1 + z_2).$$
        
        \item[(6)] If $\varphi = (t_1 \cdot t_2 = t)$, then 
        $$\varphi^* := \exists z_1 \exists z_2 \exists z_3 ((z_1 = t_1)^* \wedge (z_2 = t_2)^*  \wedge (z_3 = t)^* \wedge z_3 = z_1 \cdot z_2).$$
        
        \item[(7)] Let $\varphi = (t_1^{t_2} = t)$. As in the case (4), define $\varphi^*$ as
        \begin{align*}
            &\big( (t_1 = 0)^* \wedge (t_2 = 0)^* \wedge (1 = t)^* \big)  \vee \\
            &\big( (t_1 = 0)^* \wedge (\exists z_1\exists z_2 ((t_2 = z_1)^* \wedge z_1 z_2 = 1)) \wedge (0 = t)^* \big) \vee \\
           \exists z_1 \exists z_2 \exists z_3 \exists z & \big((z_1 = t_1)^* \wedge (z_2 = t_2)^* \wedge (\exp(z_3) = z_1) \wedge (z = t)^* \wedge z = \exp(z_2 z_3)\big).
        \end{align*}
    \end{enumerate}
    
    Similarly, we define the translation of atomic formulas of the form $t_1\leqslant t_2$. It remains for us to define the translation for the formulas of the form $(\varphi_1 \wedge \varphi_2)$, $(\varphi_1\vee \varphi_2)$: 
    $$(\varphi_1 \wedge\varphi_2)^* := (\varphi_1^* \wedge \varphi_2^*), \quad (\varphi_1 \vee \varphi_2)^* := (\varphi_1^* \vee \varphi_2^*).$$ 
    
    Further, the formula $\varphi^*$ is equivalent to $\exists$-formula (it is obvious from the construction) without occurrences of $\neg$ and $\to$. Every atomic formula of the form $t_1 \leqslant t_2$ is equivalent to $\exists z (t_1 + z^2 = t_2)$ modulo $\mathsf{ExpField + KTB}$. Formula $(t_1 = t_2) \wedge (s_1 = s_2)$ is equivalent to a formula $(t_1 - t_2)^2 + (s_1 - s_2)^2 = 0$ and $(t_1 = t_2) \vee (s_1 = s_2)$ is equivalent to a formula $(t_1 - t_2)(s_1 - s_2) = 0$ modulo $\mathsf{ExpField}$. So, $\varphi^*$ is equivalent to the formula of the form $\exists \overline y (t = 0)$. For more details, see \cite[Proposition 4.5.4]{servi_thesis}. Hence, by $\mathsf{KTB}$, the set 
    $$X_\varphi(\overline a) := \{x \in R \mid  (\mathcal R, \exp) \vDash \neg \varphi^*(x, \overline a) \wedge (x > 0)\}$$ is a finite union of intervals and points. Now the proof can be finished as those of \autoref{theorem_M_exp_ip_ExpField+MaxVal}.
\end{proof}

In order to slightly strengthen \autoref{theorem_M_exp_ip_ExpField+MaxVal} and \autoref{theorem_(M, x^y)_ip_R_exp}, let us introduce theories $\mathsf{LOpen}(\exp)$ and $\mathsf{LOpen(x^y)}$: $\mathsf{LOpen}(\dots)$ is obtained from $\mathsf{IOpen}(\dots)$ by replacing the induction scheme with the least element scheme for quantifier-free formulas in the corresponding language. Clearly, $\mathsf{LOpen}(\exp)$ implies $\mathsf{IOpen}(\exp)$ and $\mathsf{LOpen}(x^y)$ implies $\mathsf{IOpen}(x^y)$ by the standard argument. It is unknown whether the converse holds.

\begin{proposition}\label{prop_strengthening}
    Both the theorems above can be strengthened by replacing $\mathsf{IOpen}(\exp)$ by $\mathsf{LOpen}(\exp)$ and  $\mathsf{IOpen(x^y)}$ by $\mathsf{LOpen(x^y)}$ respectively.
\end{proposition}

\begin{proof}
    It suffices to notice that in all proofs above the set $X_{\varphi}(\overline a) \cap M^+$ has the least element. Thereby the least element scheme for quantifier-free formulas holds in $(\mathcal{M}^+, \exp)$ (or in $(\mathcal{M}^+, x^y)$).
\end{proof}

\begin{lemma}\label{lemma_for_Bernoulli}
    $\mathsf{ExpField} + \forall x(\exp(x) \geqslant 1 + x) \vdash \forall x \forall y \geqslant 1 (\exp(x y) \geqslant 1 + y (\exp(x) - 1))$.
\end{lemma}

\begin{remark}
    Essentially, the sentence $\forall x \forall y \geqslant 1 (\exp(x y) \geqslant 1 + y (\exp(x) - 1))$ is equivalent to the Bernoulli inequality: substituting $\log (1 + r)$ instead of $x$, where $r > - 1$, we obtain $(1 + r)^y \geqslant 1 + ry$.
\end{remark}

\begin{proof}
    We will reason inside $\mathsf{ExpField} + \forall x(\exp(x) \geqslant 1 + x)$. Let $x$ be arbitrary and $y \geqslant 1$. For $y = 1$ the inequality is trivial, so, consider the case of $y > 1$. We have that $\exp(xy - x) \geqslant 1 + xy - x$, hence, $\exp(xy) \geqslant \exp(x)(1 + xy - x)$. We claim that $\exp(x)(1 + xy - x) \geqslant 1 + y (\exp(x) - 1)$. Indeed, 
    \begin{align*}
        \exp(x)(1 + xy - x) \geqslant 1 + y (\exp(x) - 1) & \iff \\
        \exp(x)(1 + xy - x - y) \geqslant 1 - y & \iff \\
        \exp(x)(x - 1)(y - 1) \geqslant 1 - y & \iff \\
        \exp(x)(x - 1) \geqslant -1 & \iff \\
        x - 1 \geqslant -\exp(-x) & \iff \\
        \exp(-x) \geqslant 1 - x
    \end{align*}
    and the latter is true. So, $\exp(x y) \geqslant 1 + y (\exp(x) - 1)$.
\end{proof}

\begin{proposition}\label{proposition_ip_of_exp_rcf}
    Let $\mathcal{M}$ be a discretely ordered ring and $x^y\colon  M^+\times M^+ \to M^+$. If there is an exponential field $(\mathcal{R}, \exp)$ such that $\mathcal{M} \subseteq^{\IP}_{x^y} (\mathcal{R}, \exp)$, $(\mathcal{R}, \exp) \vDash \mathsf{ExpField} + \mathsf{RCF} +  \forall x(\exp(x) \geqslant 1 + x)$ and, moreover, the function $x^y$ coincides with a power function induced by $\exp$ (see \cref{def_x^y}), then $(\mathcal{M}^+, x^y)$ is a model of $\mathsf{IOpen} + \mathsf{T_{x^y}}$.
\end{proposition}

\begin{proof}
    Let $\mathcal{M} \subseteq^{\IP}_{x^y} (\mathcal{R}, \exp)$ and $(\mathcal{R}, \exp) \vDash \mathsf{ExpField} + \mathsf{RCF} + \forall x(\exp(x) \geqslant 1 + x)$. By \cref{shepherdson_theorem_reformulated} $\mathcal{M}^+ \vDash \mathsf{IOpen}$.

    It is straightforward to verify that $(\mathcal M^+, x^y) \vDash$ (T1)--(T8). We have $(\mathcal M^+, x^y) \vDash$ (T11) by \cref{lemma_for_Bernoulli} and remark after it.  We verify (T9) and (T10).
    
    Let $x \in M^{>0}$, $y := \big[\frac{\log x}{\log 2}\big]$, where $\log = \exp^{-1} \colon  R^{>0} \to R$ and $[r]$ denotes an integer part of $r \in R$. It is clear that $y \in M^+$. Since $y \leqslant \frac{\log x}{\log 2}$, 
    $$2^y = \exp(y \log 2) \leqslant \exp(\log x) = x.$$ Since $\frac{\log x}{\log 2} < y + 1$, 
    $$x = \exp(\log x) < \exp((y + 1)\log 2) = 2^{y + 1}.$$ That is, $(\mathcal{M}^+, x^y) \vDash$ (T9). In a similar way one can prove that $(\mathcal{M}^+, x^y) \vDash$ (T10), just put $z := [x^{1/y}]$. So, we have proved that $(\mathcal M^+, x^y) \vDash \mathsf{IOpen + T_{x^y}}$.
    
    In order to prove the opposite implication we construct a real closed exponential field $(K_\mathcal M, \exp_\mathcal M)$ containing given $(\mathcal M^+, x^y) \vDash \mathsf{IOpen + T_{x^y}}$ as an $x^y$-integer part. This construction is presented in \cref{section_construction_of_K_M}.
\end{proof}

\begin{remark}
    $\mathsf{T_{x^y}}$ is not very strong, as the following proposition shows. 
\end{remark}

\begin{proposition}\label{prop_iopen(x^y)|--TT_x^y}
    $\mathsf{IOpen(x^y)} \vdash \mathsf{T_{x^y}}$.
\end{proposition}

\begin{proof}
    Axioms (T1)--(T10) follow from $\mathsf{IOpen(x^y)}$ easily by induction. We explain how to prove (T11) (the Bernoulli inequality), which reads as 
    $$(x > 0 \wedge y > 0) \to \Big( \big(\frac{x}{y}\big)^z \geqslant 1 + z\big(\frac{x}{y} - 1\big)\Big).$$
    Fix some $(\mathcal M^+, x^y) \vDash \mathsf{IOpen(x^y)}$ and $x, y \in M^+$. We prove the inequality by induction on $z$. If $z = 0$, the Bernoulli inequality holds.

    Suppose $\big(\frac{x}{y}\big)^z \geqslant 1 + z\big(\frac{x}{y} - 1\big)$. Then we have 
    $$\big(\frac{x}{y}\big)^{z + 1} \geqslant \big(1 + z(\frac{x}{y} - 1)\big)\frac{x}{y} = \frac{x}{y} + z\big(\frac{x^2}{y^2} - \frac{x}{y}\big) = $$
    $$ = \frac{x}{y} + z \big(\frac{x^2}{y^2} - \frac{2x}{y} + 1\big) + z \frac{x}{y} - z = \frac{x}{y} + z \big(\frac{x}{y} - 1\big)^2 + z \frac{x}{y} - z \geqslant$$
    $$ \geqslant \frac{x}{y} + z \frac{x}{y} - z = 1 + (z + 1)(\frac{x}{y} - 1).$$
    By the induction axiom we have $\forall z \Big( \big(\frac{x}{y})^z \geqslant 1 + z(\frac{x}{y} - 1\big)\Big)$. So, the inequality holds for all $z$ and $(\mathcal M^+, x^y) \vDash$ (T11).
\end{proof}

\section{Construction of the exponential field $(\mathcal K_{\mathcal{M}}, \exp_{\mathcal{M}})$}\label{section_construction_of_K_M}

This section is mainly devoted to the proof of the following result.

\begin{proposition}\label{proposition_exp_rcf_from_M}
    Let $\mathcal{M}$ be a discretely ordered ring, $x^y\colon  M^+ \times M^+ \to M^+$ and $(\mathcal{M}^+, x^y)$ is a model of $\mathsf{IOpen} + \mathsf{T_{x^y}}$. Then there is an exponential field $(\mathcal K_{\mathcal{M}}, \exp_{\mathcal{M}})$ such that $\mathcal{M} \subseteq^{\IP}_{x^y} (\mathcal K_{\mathcal{M}}, \exp_{\mathcal{M}})$, $(\mathcal K_{\mathcal{M}}, \exp_{\mathcal{M}}) \vDash \mathsf{ExpField} + \mathsf{RCF} +  \forall x(\exp(x) \geqslant 1 + x)$ and, moreover, the function $x^y$ coincides with a power function induced by $\exp_\mathcal M$ (see \cref{def_x^y}).
\end{proposition}

Although our construction differs from the one in \cite{Carl_Krapp_2021} (he used only sequences definable in $\mathcal M^+$), some of the proofs from his paper can also be applied to our construction. In such cases, we will refer to his paper. Now we proceed with the proof of \cref{proposition_exp_rcf_from_M}.

Let us fix a discretely ordered ring $\mathcal M$ and $x^y\colon  M^+ \times M^+ \to M^+$ such that $(\mathcal{M}^+, x^y) \vDash \mathsf{IOpen} + \mathsf{T_{x^y}}$. We denote by $\mathcal{F}(\mathcal{M})$ the ordered quotient field of $\mathcal{M}$ and by $F(\mathcal M)$ its domain. We call a \textit{rational $\mathcal{M}$-sequence} a function $a\colon  M^+ \to {F}(\mathcal{M})$, $a(n)$ will be denoted by $a_n$, and a sequence $n \mapsto a_n$ by $(a_n)$. A rational $\mathcal{M}$-sequence $a$ is called an \textit{$\mathcal{M}$-Cauchy sequence} if the following condition is satisfied:
$$\forall k \in {M}^{>0} \exists N \in {M}^+ \forall n, m \in {M}^+ \big(n, m > N \implies |a_n - a_m| < \frac{1}{k}\big).$$
Let us introduce an equivalence relation on $\mathcal{M}$-Cauchy sequences: $a \sim b$ if
$$\forall k \in {M}^{>0} \exists N \in {M}^+ \forall n > N \big(|a_n - b_n| < \frac{1}{k}\big).$$

Denote by $K_\mathcal M$ the set of equivalence classes of all $\mathcal M$-Cauchy sequences modulo $\sim$. Now introduce the operations and the order relation on $K_{\mathcal{M}}$ (where $[a]$ denotes the equivalence class of $a$):
$$[a] +_{K_\mathcal M} [b] := [a + b],$$
$$-_{K_\mathcal M}[a] := [-a]$$
$$[a] \cdot_{K_\mathcal M} [b] := [a \cdot b],$$
$$[a] <_{K_\mathcal M} [b] :\iff \exists k \in {M}^{>0}\exists N \in {M}^+ \forall n> N (a_n + \frac{1}{k}<b_n),$$
$$[a] \leqslant_{K_\mathcal M} [b] :\iff ([a] <_{K_\mathcal M} [b] \vee [a] = [b]).$$
For $q \in F(\mathcal M)$ let $(q)$ denote the $\mathcal M$-Cauchy sequence $n \mapsto q$. It is easy to check that the following statement holds:
\begin{proposition}
    The introduced operations are well-defined and 
    $$\mathcal K_{\mathcal{M}} := (K_{\mathcal{M}}, +_{K_\mathcal M}, -_{K_\mathcal M}, \cdot_{K_\mathcal M}, [(0)], [(1)], \leqslant_{K_\mathcal M})$$ 
    is an ordered field. Moreover, $q \mapsto [(q)]$ is an embedding of ordered fields.
\end{proposition}
If there is no confusion, we will write $+$, $\cdot$ and $\leqslant$ instead of $+_{K_\mathcal M}$, $\cdot_{K_\mathcal M}$, and $\leqslant_{K_\mathcal M}$, $(a_n)$ instead of $[(a_n)]$ and $q$ instead of $[(q)]$. Also, we will think of $\mathcal F(\mathcal M)$ as a subfield of $\mathcal K_\mathcal M$. We will not use the notation $[a]$ for the equivalence class further, but we will use it for the integer part of $a$.
\begin{proposition}\label{prop_M_ip_K_M}
    $\mathcal{M} \subseteq^{\IP} \mathcal K_{\mathcal{M}}$.
\end{proposition}
\begin{proof}
    First, we note that $\mathcal{M} \subseteq^{\IP} \mathcal{F}(\mathcal{M})$ by \autoref{shepherdson_theorem}. Now, let $a \in K_{\mathcal{M}}$. By definition, 
    $$\exists N \in {M}^+ \forall n \in {M}^+ \big(n > N \rightarrow |a_n - a_{N + 1}| < \frac{1}{2}\big).$$ Since $\mathcal{M} \subseteq^{\IP} \mathcal{F}(\mathcal{M})$, $\exists m \in {M}$ such that $m \leqslant a_{N + 1} < m + 1$. Then $m - \frac{1}{2} \leqslant a \leqslant m + \frac{3}{2}$. So one of the $m - 1,\ m,\ m+1$ is an integer part of $a$.
\end{proof}

\begin{lemma}[{\cite[Lemma 7.9]{krapp:2019}}]\label{lemma_for_density_of_F(M)}
    If $\mathcal M \subseteq^{\IP} \mathcal L$ for some ordered field $\mathcal L$, then $\mathcal F(\mathcal M)$ is dense in $\mathcal L$.
\end{lemma}

Let us denote by $\mathcal R(\mathcal M)$ the real closure of $\mathcal F(\mathcal M)$.

\begin{corollary}\label{corollary_F(M)_is_dense_in_K_M}
    $\mathcal{F}(\mathcal{M})$ is dense in $\mathcal K_{\mathcal{M}}$ and $\mathcal R(\mathcal M)$.
\end{corollary}

\begin{proof}
    We have $\mathcal{M} \subseteq^{\IP} \mathcal K_{\mathcal{M}}$ by \cref{prop_M_ip_K_M} and $\mathcal M \subseteq^{\IP} \mathcal R (\mathcal M)$ by \autoref{shepherdson_theorem}, hence, it remains to apply \cref{lemma_for_density_of_F(M)}.
\end{proof}

\begin{proposition}\label{prop_R(M)_subfield_K_M}
    $\mathcal R(\mathcal{M})$ can be embedded into $\mathcal K_\mathcal M$ (with $\mathcal F(\mathcal M)$ fixed).
\end{proposition}

\begin{proof}
    By \cref{corollary_F(M)_is_dense_in_K_M}, $\mathcal F(\mathcal M)$ is dense in $\mathcal R(\mathcal M)$. So, for $r \in R(\mathcal M)$, choose an $q_n \in F(\mathcal M)$ such that $|r - q_n| < \frac{1}{n + 1}$ for all $n \in M^+$ (such $q_n$ exists since $(r, r + \frac{1}{n + 1}) \cap F(\mathcal M) \ne \varnothing$). Clearly, a map $r \mapsto (q_n)$ is an embedding of $\mathcal R(\mathcal M)$ into $\mathcal K_\mathcal M$ with $\mathcal F(\mathcal M)$ fixed pointwise.
\end{proof}

By the proposition above, $\mathcal R(\mathcal M)$ can be considered as a dense subfield of $\mathcal K_\mathcal M$.

\begin{definition}[\cite{scott1969}]
    An ordered field $\mathcal L$ is called \emph{complete} if it has no proper ordered field extension containing $\mathcal L$ as a dense subfield. If $\mathcal L \subseteq \mathcal L'$ is an ordered field extension, then $\mathcal L'$ is called a \emph{completion} of $\mathcal L$ if $\mathcal L'$ is complete and $\mathcal L$ is dense in $\mathcal L'$.
\end{definition}

\begin{theorem}[\cite{scott1969}]\label{theorem_on_order_completeness}
    For any ordered field $\mathcal L$ there is a completion of $\mathcal L$ and it is unique up to an isomorphism fixing $L$ pointwise. Moreover, the completion of a real closed field is real closed.
\end{theorem} 

\begin{proposition}\label{prop_on_completion}
    $\mathcal K_\mathcal M$ is a completion of $\mathcal R(\mathcal M)$.
\end{proposition}

\begin{proof}
    By \cref{prop_R(M)_subfield_K_M} it is sufficient to show that $\mathcal K_\mathcal M$ is complete. Aiming at a contradiction, assume that there is an ordered field $\mathcal L \supsetneq \mathcal K_\mathcal M$ such that $\mathcal K_\mathcal M$ is dense in $\mathcal L$. Note that since $\mathcal F(\mathcal M)$ is dense in $\mathcal K_\mathcal M$, it is dense in $\mathcal L$. Fix an arbitrary $a \in L \setminus K_\mathcal M$. Given $n \in M^+$ choose $q_n \in (a, a + \frac{1}{n + 1}) \cap F(\mathcal M)$ and $q'_n \in (a - \frac{1}{n + 1}, a) \cap F(\mathcal M)$ (these sets are nonempty by density). Clearly, $n \mapsto q_n$ and $n \mapsto q_n'$ are $\mathcal M$-Cauchy sequences, denote it by $r, r' \in K_\mathcal M$ its equivalence classes respectively. By the assumption, $r, r' \ne a$ and, hence, $r \ne r'$. Clearly, it implies that $r' < a < r$. But $r = r'$ since $|q_n - q_n'| < \frac{2}{n + 1}$, a contradiction.
\end{proof}

\begin{corollary}\label{corollary_K_M_is_rcf}
    $\mathcal K_{\mathcal{M}}$ is a real closed field. 
\end{corollary}
    
\begin{proof}
    Follows from \autoref{theorem_on_order_completeness} and \cref{prop_on_completion}.
\end{proof}

Now we need to define an exponentiation on $\mathcal K_{\mathcal{M}}$. First define base-2 exponentiation $\exp_2\colon  {F}(\mathcal{M})\to K_{\mathcal{M}}^{>0}$. Let $n, b, c \in {M}^{+}$, $c > 0$. Let 
$$B(n, b, c) = \{m \in {M}^+\mid m^c \leqslant 2^{n c + b}\}.$$ 
By (T10) there exists a maximum in $B(n, b, c)$. Let $d_n = \max B(n, b, c)$.
By definition, $$\frac{d_n^c}{2^{n c}} \leqslant 2^b < \frac{(d_n + 1)^c}{2^{n c}}.$$

Now define $\exp_2\big(\frac{b}{c}\big)$ as $\big(\frac{d_n}{2^n}\big)$. For $a < 0$ define $\exp_2(a)$ as $(\exp_2(-a))^{-1}$.

\begin{remark}
     Under such definition $\exp_2(b) = \exp_2(\frac{b}{1}) = 2^b$, since for $c = 1$ we have $d_n = 2^{n + b}$. In particular, $\exp_2(0) = 1$. 
\end{remark}

\begin{lemma}\label{lemma_d_n/2^n_increases}
    For all $b, c \in M^{>0}$, the $\mathcal{M}$-sequence $(\frac{d_n}{2^n})$ increases and $(\frac{d_n + 1}{2^n})$ decreases.
\end{lemma}

\begin{proof}
    Let $m > n \in {M}^+$. Then 
    $$\frac{(d_n 2^{m - n})^c}{2^{m c}} = \frac{d_n^c}{2^{n c}} \leqslant 2^b.$$ Since $d_m$ is the greatest number with the property $\frac{d_m^c}{2^{m c}} \leqslant 2^b$, we get $d_n 2^{m - n}\leqslant d_m$, so $$\frac{d_n}{2^n}\leqslant\frac{d_m}{2^m}.$$ The second part of the statement can be proved in a similar way (just observe that $d_m + 1$ is the least with the property $2^b < \frac{(d_m + 1)^c}{2^{mc}}$).
\end{proof}

\begin{proposition}
    $\exp_2\big(\frac{b}{c}\big)$ is well-defined.
\end{proposition}

\begin{proof}
    Clearly, it is enough to prove the claim for $\frac{b}{c} > 0$.
    
    First, we prove that $\bigl(\frac{d_n}{2^n}\bigr)$ is an $\mathcal{M}$-Cauchy sequence. Let $m > n \in {M}^+$. By \cref{lemma_d_n/2^n_increases} we have $\frac{d_n}{2^n}\leqslant\frac{d_m}{2^m}$, by definition we have $$\frac{d_m^c}{2^{mc}} \leqslant 2^b < \frac{(d_n + 1)^c}{2^{nc}}$$ and, hence, $\frac{d_m}{2^m} < \frac{d_n + 1}{2^n}$. So, 
    $$|\frac{d_n}{2^n} - \frac{d_m}{2^m}| < \frac{1}{2^n} < \frac{1}{n}$$ (the latter inequality follows from the Bernoulli inequality (T11)). So $\bigl(\frac{d_n}{2^n}\bigr)$ is an $\mathcal{M}$-Cauchy sequence.
    
    Now we prove that the result does not depend on the choice of the numerator and denominator. To do this, it is sufficient to notice that for any $l \in {M}^{>0}$,
    $$\frac{m^{l c}}{2^{n c l}} \leqslant 2^{b l} <\frac{(m+1)^{l c}}{2^{n c l}} \iff \frac{m^{c}}{2^{n c}} \leqslant 2^b < \frac{(m + 1)^{c}}{2^{n c}}.$$ It follows that $\max B(n, b, c) = \max B(n, b l, c l)$.
\end{proof}

\begin{lemma}\label{lemma_(n+1)^l/n^l->1}
    For all $l \in {M}^{>0}$, the $\mathcal{M}$-sequence $\Big(\frac{(n + 1)^l}{n^l}\Big)$ is an $\mathcal{M}$-Cauchy sequence and is equivalent to the sequence $(1)$, that is, $\Big[\Big(\frac{(n + 1)^l}{n^l}\Big)\Big] = 1$.
\end{lemma}

\begin{proof}
    Let $n, l \in M^+$. By the Bernoulli inequality (T11) we have 
    $$\frac{n^l}{(n + 1)^l} = \Big(1 - \frac{1}{n + 1}\Big)^l \geqslant 1 - \frac{l}{n + 1},$$ so, 
    $$\frac{(n + 1)^l}{n^l} \leqslant \frac{1}{1 - \frac{l}{n + 1}} = \frac{n + 1}{n + 1 - l} = 1 + \frac{l}{n + 1 - l}$$ 
    (for sufficiently large $n$). Note that $\frac{(n + 1)^l}{n^l} \geqslant 1$, so $\Big(\frac{(n + 1)^l}{n^l}\Big)$ is equivalent to the sequence $(1)$.
\end{proof}

\begin{proposition}\label{prop_exp_is_monotone_homomorphism_from_F(M)}
    $\exp_2$ is an order-preserving embedding of an additive group $({F}(\mathcal{M}), +)$ into $(K_{\mathcal{M}}^{>0}, \cdot)$.
\end{proposition}

\begin{proof}
    \cite[Lemma 7.23]{krapp:2019}.
\end{proof} 

We will say that a sequence $f\colon  M^+ \to K_\mathcal M$ \textit{tends} to $b \in K_{\mathcal{M}}$ if 
$$\forall k \in M^{>0} \exists N \in M^+ \forall n > N\Big(|f_n - b| < \frac{1}{k}\Big).$$ Notation: $\lim\limits_{n \to \infty} f_n = b$.

\begin{lemma}\label{lemma_mth_root_respect_order}
    Let $b, m \in {M}^{>0}$, $a, c \in {F}(\mathcal{M})^{>0}$. Then 
    $$c < \exp_2\big(\frac{b}{m}\big) < a \iff c^m < 2^b < a^m.$$
\end{lemma}

\begin{proof}
    Suppose $c^m < 2^b <a^m$. Let 
    $$d_n := \max B(n, b, m) = \max \{d \in {M}^+\mid d^m \leqslant 2^{n m + b}\}.$$ Then, for any $n \in {M}^+$, 
    $$\frac{d_n^m}{2^{n m}}\leqslant 2^b < \frac{(d_n+1)^m}{2^{n m}}.$$ 
    It is clear that $d_n \geqslant 2^n$, so by the Bernoulli inequality $d_n > n$. On the other hand, it is easy to see that $d_n \leqslant 2^{n + b}$. \cref{lemma_(n+1)^l/n^l->1} implies that $\lim\limits_{n\to\infty}\frac{(d_n + 1)^m}{d_n^m} = 1$. Then we have 
    $$|2^b - \frac{(d_n+1)^m}{2^{n m}}| \leqslant |\frac{d_n^m}{2^{n m}} - \frac{(d_n+1)^m}{2^{n m}}| = \frac{d_n^m}{2^{n m}}| 1 - \frac{(d_n+1)^m}{d_n^m}| \leqslant 2^{m b} | 1 - \frac{(d_n+1)^m}{d_n^m}|,$$
    hence, $\lim\limits_{n \to \infty} \frac{(d_n+1)^m}{2^{n m}} = 2^b$ and there is an $n \in {M}^+$ such that $\frac{(d_n+1)^m}{2^{n m}} < a^m$. Hence, $\frac{d_n + 1}{2^n} < a$. Since $\frac{d_n + 1}{2^n}$ is decreasing, $\exp_2(\frac{b}{m}) < a$. Similarly, it can be proved that $c < \exp_2(\frac{b}{m})$.
    
    It remains for us to prove the opposite implication. Assume, for example, that $a^m \leqslant 2^b$. Arguing similarly to the previous, we obtain $a \leqslant \exp_2(\frac{b}{m})$, a contradiction.
\end{proof}

\begin{lemma}\label{lemma_exp(1/n) < 1 + 1/n}
    $\forall n \in M^{>0} \exp_2(\frac{1}{n}) \leqslant 1 + \frac{1}{n}$.
\end{lemma}

\begin{proof}
    By the Bernoulli inequality $2^1 = 2 \leqslant (1 + \frac{1}{n})^n$. Then by \cref{lemma_mth_root_respect_order} $\exp_2(\frac{1}{n}) \leqslant 1 + \frac{1}{n}$.
\end{proof}

Now define $\exp_2$ on all $K_{\mathcal{M}}$. For $(a_n) = (\frac{b_n}{c_n}) \in K_{\mathcal{M}}$, $(a_n) > 0$, let $\exp_2(a) =\Big(\frac{\max B(n, b_n, c_n)}{2^n}\Big)$, $\exp_2(0) := 1$, $\exp_2(-a) := (\exp_2(a))^{-1}$.

\begin{proposition}
    $\exp_2$ is well-defined.
\end{proposition}

\begin{proof}
    Let $a = (a_n) = (\frac{b_n}{c_n}) \in K_{\mathcal{M}}, a > 0$, $d_n := \max B(n, b_n, c_n)$. First, we prove that $(\frac{d_n}{2^n})$ is an $\mathcal M$-Cauchy sequence (then $\lim\limits_{n \to \infty} \frac{d_n}{2^n} = (\frac{d_n}{2^n})$).
    
    Fix $k \in {M}^{>0}$. Choose $L \in M^+$ such that $L > 3\exp_2([a] + 1)k$. Then choose an $N \in M^+$ such that $N > L$ and, for all $m \geqslant n > N$, $|a_n - a_m| < \frac{1}{L}$ and $a_n < [a] + 1$.
    
    Fix arbitrary $m \geqslant n > N$. Without loss of generality $a_n \leqslant a_m$, thereforce 
    $$|\exp_2(a_m) - \exp_2(a_n)| = \exp_2(a_m) - \exp_2(a_n) < \exp_2\big(a_n + \frac{1}{L}\big) - \exp_2(a_n) = $$ $$ = \exp_2(a_n)\Big(\exp_2\big(\frac{1}{L}\big) - 1\Big) < \exp_2([a] + 1)\Big(\exp_2(\frac{1}{L}) - 1\Big).$$
    By \cref{lemma_exp(1/n) < 1 + 1/n} $\exp_2(\frac{1}{L}) \leqslant 1 + \frac{1}{L}$, therefore 
    $$|\exp_2(a_m) - \exp_2(a_n)| < \exp_2([a] + 1)\frac{1}{L} < \frac{1}{3 k}.$$ Now
    $$\left|\frac{d_n}{2^n} - \frac{d_m}{2^m}\right| \leqslant \left|\frac{d_n}{2^n} - \exp_2(a_n)\right| + \Big|\exp_2(a_n) - \exp_2(a_m)\Big| + \left|\exp_2(a_m) - \frac{d_m}{2^m}\right| < $$ 
    $$ < \frac{1}{2^n} + \frac{1}{3 k} + \frac{1}{2^m} < \frac{1}{k}.$$
    This means that $(\frac{d_n}{2^n})$ is an $\mathcal{M}$-Cauchy sequence. 
    
    Now we prove that the result does not depend on the choice of the sequence from the equivalence class. Let $(r_n) = \Big(\frac{s_n}{c_n}\Big) \sim(a_n)$ (we can assume that the denominators in both sequences are the same, since $\max B(n, b_n, c_n) = \max B(n, l b_n, l c_n)$ for any $l \in {M}^{>0}$). Let $u_n = \max B(n, s_n, c_n)$, $D = \lim\limits_{n \to \infty}\frac{d_n}{2^n}$, $U = \lim\limits_{n \to \infty}\frac{u_n}{2^n}$. We need to prove that $D = U$. 
    
    If for any $N \in {M}^+$ there are $n, m > N$ such that $u_n \leqslant d_n$ and $d_m \leqslant u_m$, then $U \leqslant D \leqslant U$ and $U = D$.
    
    Now consider the case when $\exists N \in {M}^+$ such that $\forall n > N \: d_n < u_n$ (the case of the opposite inequality is similar). In this case, $a_n < r_n$ for all $n > N$.
    
    Fix a $k \in {M}^{>0}$. Find an $N' \geqslant N$ such that for all $n > N'$, $0 < r_n - a_n < \frac{1}{k}$. Then, for all $n > N'$, $0 < s_n - b_n < \frac{c_n}{k}$ and
    $$\Big(\frac{u_n}{2^n}\Big)^{c_n} \leqslant2^{s_n} < \exp_2(b_n + \frac{c_n}{k}) = \exp_2(b_n)\exp_2(\frac{c_n}{k}) \leqslant $$ $$ \leqslant \exp_2(\frac{c_n}{k}) \Big(\frac{d_n + 1}{2^n}\Big)^{c_n}.$$
    By \cref{lemma_mth_root_respect_order}, $\frac{u_n}{2^n} \leqslant\exp_2(\frac{1}{k}) \frac{d_n+1}{2^n}$. Hence, using \cref{lemma_exp(1/n) < 1 + 1/n}, 
    $$\frac{u_n}{2^n} \leqslant \exp_2\Big(\frac{1}{k}\Big) \frac{d_n + 1}{2^n} \leqslant \Big(1 + \frac{1}{k}\Big) \frac{d_n + 1}{2^n},$$
    and
    $$\left|\frac{u_n}{2^n} - \frac{d_n}{2^n}\right| = \frac{u_n}{2^n} - \frac{d_n}{2^n} < \Big(1 + \frac{1}{k}\Big)\frac{d_n + 1}{2^n} - \frac{d_n}{2^n} = \frac{1}{2^n} + \frac{d_n + 1}{2^n k}.$$
    So, since $\Big(\frac{d_n + 1}{2^n}\Big)$ is bounded, $\Big(\frac{d_n}{2^n}\Big) \sim \Big(\frac{u_n}{2^n}\Big)$, and $D = U$.
\end{proof}

\begin{lemma}\label{lemma_exp(lim a_n) = lim exp(a_n)}
    Let $a = (a_n) \in K_{\mathcal{M}}$. Then $\lim\limits_{n \to \infty} \exp_2(a_n) = \exp_2(a)$.
\end{lemma}

\begin{proof}
    It is enough to consider the case of $a > 0$. Fix $k \in {M}^{>0}$, $d_n := \max B(n, b_n, c_n)$, where $a_n = \frac{b_n}{c_n}$. There exists some $N \in {M}^+$ such that for all $n > N$ hold $|\frac{d_n}{2^n} - \exp_2(a)| < \frac{1}{2 k}$ and $\frac{1}{2^n} < \frac{1}{2 k}$. Then for all $n > N$ 
    $$|\exp_2(a_n) - \exp_2(a)| \leqslant \Big|\exp_2(a_n) - \frac{d_n}{2^n}\Big| + \Big|\frac{d_n}{2^n} - \exp_2(a)\Big| < \frac{1}{2^n} + \frac{1}{2 k} < \frac{1}{k}.$$
\end{proof}

\begin{theorem}
    $(\mathcal K_{\mathcal{M}}, \exp_2)$ is an exponential field.
\end{theorem}

\begin{proof}
    First, we prove that $\exp_2$ is an order-preserving embedding from $(K_\mathcal M, +)$ to $(K_\mathcal M^{>0}, \cdot)$. Clearly, for all $a \in K_\mathcal M$, $\exp_2 (a) > 0$. Let $a, b \in K_{\mathcal{M}}, a = (a_n), b = (b_n)$. Then $\exp_2(a_n +b_n) = \exp_2(a_n) \exp_2(b_n)$ by \cref{prop_exp_is_monotone_homomorphism_from_F(M)}. By \cref{lemma_exp(lim a_n) = lim exp(a_n)} $$\exp_2(a + b) = \lim\limits_{n \to \infty} \exp_2(a_n + b_n) = \lim\limits_{n \to \infty} \exp_2(a_n) \exp_2(b_n) = \exp_2(a)\exp_2(b).$$
    
    Let $a =(a_n), b =(b_n) \in K_{\mathcal{M}}, a <b$. There exist $k, N\in {M}^{>0}$ such that for all $n > N$, $a_n +\frac{1}{k} < b_n$. Then,
    $$\exp_2(a) = \lim\limits_{n \to \infty} \exp_2(a_n) < \exp_2\big(\frac{1}{k}\big) \lim\limits_{n \to \infty} \exp_2 (a_n) = \lim\limits_{n \to \infty} \exp_2\big(a_n + \frac{1}{k}\big) \leqslant$$
    $$\leqslant \lim\limits_{n \to \infty} \exp_2(b_n) = \exp_2(b).$$
    
    It remains to prove that $\exp_2$ is onto. 
    
    Let $a = (a_n) \in K_{\mathcal{M}}^{>1}$, $n \in {M}^+$, this easily implies the case of $0 < a < 1$. Since $a < [a] + 1$, we may assume that for all $n \in M^+$, $a_n < [a] + 1$.  By (T9) there is $d_n \in {M}^+$ such that $2^{d_n} \leqslant [a_n^{2^n}] < 2^{d_n + 1}$. For such a $d_n$ we have $2^{d_n} \leqslant a_n^{2^n} < 2^{d_n + 1}$. By \cref{lemma_mth_root_respect_order}, $\exp_2(\frac{d_n}{2^n}) \leqslant a_n <  \exp_2(\frac{d_n + 1}{2^n})$. Also $2^{d_n} \leqslant a_n^{2^n} < 2^{a_n 2^n}$, so $ d_n \leqslant a_n 2^n$ and, hence, $\frac{d_n}{2^n} < [a] + 1$. Then 
    $$0 < \exp_2\big(\frac{d_n + 1}{2^n}\big) - \exp_2\big(\frac{d_n}{2^n}\big) = \exp_2\big(\frac{d_n}{2^n}\big)\Big(\exp_2\big(\frac{1}{2^n}\big) - 1\Big) \leqslant \frac{2^{[a] + 1}}{2^n}$$ (the latter inequality is implied by \cref{lemma_exp(1/n) < 1 + 1/n}).
    
    Suppose that $\big(\frac{d_n}{2^n}\big)$ is not an $\mathcal{M}$-Cauchy sequence. Then there is $k_0 \in {M}^{>0}$ such that for all $N \in {M}^+$, there are $n, m > N$ such that 
    $$\frac{d_{n}}{2^{n}} - \frac{d_{m}}{2^{m}} \geqslant \frac{1}{k_0}.$$ Choose $k_1$ such that 
    $$\Big(1 - \frac{1}{\exp_2(\frac{1}{k_0})}\Big)[a] > \frac{1}{k_1}$$
    (such a $k_1$ exists since the lhs of the inequality is positive). Let us fix $N$ such that $\frac{2^{[a] + 1}}{2^N} < \frac{1}{k_1}$ and for all $n, m > N$, $|a_n - a_m| < \frac{1}{2 k_1}$. Choose $n, m > N$ such that $\frac{d_n}{2^n} - \frac{d_m}{2^m} \geqslant \frac{1}{k_0}$ (such exist by our assumption). Then we have the following inequalities:
    \begin{align*}
        \exp_2\left(\frac{d_m}{2^m}\right) &\leqslant \frac{\exp_2\left(\frac{d_n}{2^n}\right)}{\exp_2\left(\frac{1}{k_0}\right)} \leqslant \frac{a_n}{\exp_2\left(\frac{1}{k_0}\right)}  \\
        &= a_n - a_n\left(1 - \frac{1}{\exp_2\left(\frac{1}{k_0}\right)}\right)  \\
        &\leqslant a_n - [a]\left(1 - \frac{1}{\exp_2\left(\frac{1}{k_0}\right)}\right)  \\
        &< a_n - \frac{1}{k_1} < a_m + \frac{1}{2k_1} - \frac{1}{k_1} 
        \\
        &= a_m - \frac{1}{2k_1} \leqslant a_m - \frac{2^{[a] + 1}}{2^m} \\
        &< \exp_2\left(\frac{d_m + 1}{2^m}\right) - \frac{\exp_2([a] + 1)}{2^m} \\
        &< \exp_2\left(\frac{d_m}{2^m}\right),
    \end{align*} so, we have got a contradiction. Hence, $(\frac{d_n}{2^n})$ is an $\mathcal{M}$-Cauchy sequence. 
    
    Since $\exp_2(\frac{d_n}{2^n}) \leqslant a_n < \exp_2(\frac{d_n}{2^n}) + \frac{2^{[a] + 1}}{2^n}$, $\lim\limits_{n \to \infty}\exp_2(\frac{d_n}{2^n}) = a$. By \cref{lemma_exp(lim a_n) = lim exp(a_n)}, $\exp_2\Big((\frac{d_n}{2^n})\Big) = a$.
\end{proof}

\begin{proposition}\label{prop_operations_are_continuous}
    $+, \cdot, \exp_2, \log_2$ are continuous (see remark after \cref{lemma_about_models_of_MaxVal(exp)} for the definition of continuity).
\end{proposition}

\begin{proof}
    Proofs of continuity $+$ and $\cdot$ are trivial (it is a general fact for ordered fields). $\exp_2$ is continuous by \cite[Proposition 2.11]{krapp:2019}. $\log_2$ is strictly increasing and surjective, hence, a preimage of any open interval is an open interval. This implies continuity immediately.  
\end{proof}

\begin{lemma}\label{lemma_(1+1/n)^n_is_incr}
    The sequence $\Big((1 + \frac{1}{2^n})^{2^n}\Big)$ is increasing.
\end{lemma}

\begin{proof}
    Let $m, n \in M^+, m > n$. Then we have 
    $$\frac{(1 + \frac{1}{2^m})^{2^m}}{(1 + \frac{1}{2^n})^{2^n}} = \left(\frac{(1 + \frac{1}{2^m})^{2^{m - n}}}{1 + \frac{1}{2^n}} \right)^{2^n} \geqslant \left(\frac{1 + \frac{1}{2^n}}{1 + \frac{1}{2^n}} \right)^{2^n} = 1,$$
    where the last inequality follows from the Bernoulli inequality.
\end{proof}

\begin{proposition}
    The sequence $\Big(\big(1 + \frac{1}{2^n}\big)^{2^n}\Big)$ is an $\mathcal M$-Cauchy sequence. We will denote this sequence by $e_\mathcal M$.
\end{proposition}

\begin{proof}
    First we argue that $\big(1 + \frac{1}{2^n}\big)^{2^n}$ is bounded. Indeed, for $n > 0$, by the Bernoulli inequality we have 
    $$\left(\frac{1}{1+\frac{1}{2^n}}\right)^{2^{n - 1}} = \left(\frac{2^n}{2^n + 1 }\right)^{2^{n - 1}} = \left(1 - \frac{1}{2^n + 1 }\right)^{2^{n - 1}} \geqslant 1 - \frac{2^{n - 1}}{2^n + 1} \geqslant \frac{1}{2}$$
    and, hence, $\big(1 + \frac{1}{2^n}\big)^{2^n} \leqslant 4$.

    Next, for all $m, n \in M^{+}$, $n  < m$, we have
    $$\frac{(1 + \frac{1}{2^m})^{2^m}}{(1 + \frac{1}{2^n})^{2^n}} = \left(\frac{(1 + \frac{1}{2^m})^{2^{m - n}}}{1 + \frac{1}{2^n}} \right)^{2^n} = \left(\frac{1}{(1 + \frac{1}{2^n})(\frac{1}{1 + \frac{1}{2^m}})^{2^{m - n}}} \right)^{2^n} = \left(\frac{1}{(1 + \frac{1}{2^n})(1 - \frac{1}{2^m + 1})^{2^{m - n}}} \right)^{2^n} \leqslant$$
    $$\leqslant \left(\frac{1}{(1 + \frac{1}{2^n})(1 - \frac{2^{m - n}}{2^m + 1})} \right)^{2^n} = \frac{1}{\big(1 + \frac{1}{2^n} - \frac{2^{m - n}}{2^m + 1} - \frac{2^{m - 2n}}{2^m + 1}\big)^{2^n}} \leqslant$$
    $$\leqslant \frac{1}{1 + 1 - \frac{2^m}{2^m + 1} - \frac{2^{m - n}}{2^m + 1}} = \frac{1}{1 + \frac{1 - 2^{m - n}}{2^m + 1}} < \frac{1}{1 - \frac{1}{2^n}} = 1 + \frac{1}{2^n - 1},$$
    where all non-trivial inequalities follow from the Bernoulli inequality.

    Now we are ready to prove that $\Big(\big(1 + \frac{1}{2^n}\big)^{2^n}\Big)$ is an $\mathcal M$-Cauchy sequence. Consider $n, m \in M^+, m > n$. Then we have 
    $$\left|\big(1 + \frac{1}{2^m}\big)^{2^m} - \big(1 + \frac{1}{2^n}\big)^{2^n}\right| = \big(1 + \frac{1}{2^m}\big)^{2^m} - \big(1 + \frac{1}{2^n}\big)^{2^n} = \big(1 + \frac{1}{2^n}\big)^{2^n} \left(\frac{\big(1 + \frac{1}{2^m}\big)^{2^m}}{\big(1 + \frac{1}{2^n}\big)^{2^n}} - 1 \right) < \frac{4}{2^n - 1},$$
    and this implies the desired.
\end{proof}

Let us define $\exp_\mathcal M$ as 
$$\exp_\mathcal M (a) := \exp_2(a \log_2 e_\mathcal M) = e_{\mathcal M}^a.$$ Clearly, $(\mathcal K_\mathcal M, \exp_\mathcal M)$ is an exponential field. We will denote by $\ln_\mathcal M$ the inverse to $\exp_\mathcal M$.

\begin{proposition}\label{prop_exp_M_is_good}
    $\exp_\mathcal M(a) \geqslant 1 + a$ for all $a \in K_\mathcal M$.
\end{proposition}

\begin{proof}
    Fix $a > 0$. First observe that for $n \in M^+$ we have 
    \begin{align*}
        \big(1 + \frac{1}{2^n}\big)^{a 2^n} \geqslant 1 + a - \frac{1}{2^n} && (*)
    \end{align*} 
    and \begin{align*}
        \big(1 - \frac{1}{2^n + 1}\big)^{a 2^n} \geqslant 1 - \frac{a + \frac{1}{2^n}}{1 + \frac{1}{2^n}}. && (**)
    \end{align*} 
    Indeed, let $m := [2^n a]$, then $\frac{m}{2^n} \leqslant a < \frac{m}{2^n} + \frac{1}{2^n}$. So, 
    $$\big(1 + \frac{1}{2^n}\big)^{a 2^n} \geqslant \big(1 + \frac{1}{2^n}\big)^{m} \geqslant 1 + \frac{m}{2^n} > 1 + a - \frac{1}{2^n}$$
    and $$\big(1 - \frac{1}{2^n + 1}\big)^{a 2^n} > \big(1 - \frac{1}{2^n + 1}\big)^{m + 1} \geqslant 1 - \frac{m + 1}{2^n + 1} \geqslant 1 - \frac{a + \frac{1}{2^n}}{1 + \frac{1}{2^n}}.$$
    
    Hence, we have
    \begin{align*}
        \exp_\mathcal M(a) & = \\
        \exp_2(a \log_2 e_\mathcal M) & = & \text{(by \cref{prop_operations_are_continuous})}\\
        \lim\limits_{n \to \infty} \exp_2(a \log_2 \big(1 + \frac{1}{2^n}\big)^{2^n}) & = \\
        \lim\limits_{n \to \infty} \big(1 + \frac{1}{2^n}\big)^{a 2^n} & \geqslant & \text{(by observation }(*))\\
        \lim\limits_{n \to \infty} \big(1 + a - \frac{1}{2^n}\big) & = \\
        1 + a
    \end{align*}
    and
    \begin{align*}
        \exp_\mathcal M(-a) & = \\
        {\exp_2(-a \log_2 e_\mathcal M)} & = & \text{(by \cref{prop_operations_are_continuous})}\\
        \lim\limits_{n \to \infty} \exp_2\big(-a \log_2 (1 + \frac{1}{2^n})^{2^n}\big) & = \\
        \lim\limits_{n \to \infty} \big(1 + \frac{1}{2^n}\big)^{-a 2^n} & = \\
        \lim\limits_{n \to \infty} \big(\frac{1}{1+\frac{1}{2^n}}\big)^{a 2^n} & = \\
        \lim\limits_{n \to \infty} \big(1 - \frac{1}{2^n + 1}\big)^{a 2^n} & \geqslant & \text{(by observation }(**))\\
        \lim\limits_{n \to \infty} \big(1 - \frac{a + \frac{1}{2^n}}{1 + \frac{1}{2^n}}\big) & = \\
        1 - a.
    \end{align*}
\end{proof}

\begin{proposition}\label{prop_for_a^b}
    $\exp_\mathcal M (b \ln_\mathcal M a) = \exp_2(b \log_2 a)$ for all $a, b \in K_\mathcal M, a > 0$.
\end{proposition}

\begin{proof}
    Let us fix $a, b \in K_\mathcal M, a > 0$. It is easy to see that  $\log_2 a = \ln_\mathcal M a \log_2 e_\mathcal M$. Hence, we have 
    $$\exp_\mathcal M (b \ln_\mathcal M a) = \exp_2(b \ln_\mathcal M a \log_2 e_\mathcal M) = \exp_2(b \log_2 a).$$
\end{proof}

\begin{proposition}\label{prop_M_x^y_ip_K_M}
    For all $m_1, m_2 \in M^{>0}$ we have $m_2^{m_1} = \exp_\mathcal M (m_1 \ln_\mathcal M(m_2))$ (and, by \cref{prop_M_ip_K_M}, $\mathcal M \subseteq^{\IP}_{x^y} (\mathcal K_{\mathcal M}, \exp_\mathcal M)$).
\end{proposition}

\begin{proof}
    By \cref{prop_for_a^b} it is enough to prove that $m_2^{m_1} = \exp_2 (m_1 \log_2 (m_2))$. 
    
    Fix $m_1, m_2 \in M^{>0}$. Choose an $\mathcal M$-Cauchy sequence $(\frac{b_n}{c_n})$ such that $\big(\frac{b_n}{c_n}\big) = \log_2(m_2)$ and $\frac{b_n}{c_n} \leqslant \log_2(m_2)$ for all $n \in M^+$ (clearly, this is possible). Then, by several applications of \cref{lemma_mth_root_respect_order} we have 
    \begin{align*}
        \frac{b_n}{c_n} \leqslant \log_2(m_2) & \iff \\
        \exp_2\Big(\frac{b_n}{c_n}\Big) \leqslant m_2 & \iff \\
        2^{m_1 b_n} \leqslant m_2^{m_1 c_n} & \iff \\
        \exp_2\Big(\frac{m_1 b_n}{c_n}\Big) \leqslant m_2^{m_1}
    \end{align*}
    and, hence, by \cref{prop_operations_are_continuous}, $\exp_2(m_1 \log_2(m_2)) \leqslant m_2^{m_1}$. Similarly one can show that the opposite inequality holds. So, $m_2^{m_1} = \exp_2(m_1 \log_2(m_2))$.
\end{proof}

Now \cref{prop_M_x^y_ip_K_M}, \cref{prop_exp_M_is_good} and \cref{corollary_K_M_is_rcf} imply \cref{proposition_exp_rcf_from_M}. Finally, \cref{proposition_ip_of_exp_rcf} and \cref{proposition_exp_rcf_from_M} combine the following result:

\begin{theorem}\label{theorem_(M,x^y)_is_iopen+T_x^y<=>ip_of_ExpField_RCF+Bern)}
    Let $\mathcal{M}$ be a discretely ordered ring and $x^y\colon  M^+ \times M^+\to M^+$. Then $(\mathcal{M}^+, x^y)$ is a model of $\mathsf{IOpen} + \mathsf{T_{x^y}}$ iff there is an exponential field $(\mathcal{R}, \exp)$ such that $\mathcal{M} \subseteq^{\IP}_{x^y} (\mathcal{R}, \exp)$, $(\mathcal{R}, \exp) \vDash \mathsf{ExpField} + \mathsf{RCF} +  \forall x(\exp(x) \geqslant 1 + x)$ and, moreover, the function $x^y$ coincides with a power function induced by $\exp$ (see \cref{def_x^y}).
\end{theorem}

As a trivial consequence, one can obtain a variant of the Bernoulli inequality for rational powers.

\begin{corollary}
    $\mathsf{IOpen + T_{x^y}} \vdash (x > 0 \wedge y > 0 \wedge z \geqslant t > 0) \to \Big( \big(\frac{x}{y}\big)^{\frac{z}{t}} \geqslant 1 + \frac{z}{t}\big(\frac{x}{y} - 1\big)\Big)$.
\end{corollary}

\begin{proof}
    Let $(\mathcal M^+, x^y)$ be a model of $\mathsf{IOpen + T_{x^y}}$ and $\mathcal M \subseteq^{\IP}_{x^y} (\mathcal R, \exp) \vDash \mathsf{ExpField} + \mathsf{RCF} + \forall x (\exp(x) \geqslant 1 + x)$. By \cref{lemma_for_Bernoulli} and remark after it we have $(\mathcal R, \exp) \vDash \forall r > -1 \forall y \geqslant 1 ((1 + r)^y \geqslant 1 + ry)$. Hence, the same holds for $(\mathcal M^+, x^y)$ for ``rational'' parameters.
\end{proof}

\section{Constructing a nonstandard model of $\mathsf{IOpen}(\exp)$ and $\mathsf{IOpen(x^y)}$}\label{section_nonstandard_model}

When constructing a nonstandard model of $\mathsf{IOpen}$, J. Shepherdson considered a real closed field of the form 
$$\{a_p t^{p/q} + a_{p - 1} t^{(p - 1)/q} + \dots + a_0 +a_{-1}t^{-1/q} + \dots \mid a_i\in R\},$$ where the field $\mathcal R$ is real closed. To build a nonstandard model of $\mathsf{IOpen}(\exp)$ and $\mathsf{IOpen(x^y)}$, we will consider a construction generalizing fields of this form. We consider an o-minimal exponential field $\mathbb{R}((t))^{\LE}$, (where LE stands for logarithmic-exponential series), find its exponential integer part $\mathcal{M}$ and apply \autoref{theorem_M_exp_ip_ExpField+MaxVal} and \autoref{theorem_(M, x^y)_ip_R_exp} to establish that $\mathcal{M}^+$ is a model of $\mathsf{IOpen}(\exp)$ and $\mathsf{IOpen(x^y)}$. 

The field $\mathbb{R}((t))^{\LE}$ and definitions used below are introduced in \cite{vandendries_macintyre_marker:2001}. Here we only describe the main steps of the construction. All the proofs can also be found in \cite{vandendries_macintyre_marker:2001}.

\begin{remark}
    In \cite{vandendries_macintyre_marker:2001} the authors use a slightly different terminology: in their paper, an ordered field with a strictly increasing \textit{homomorphism} between the additive group and multiplicative group of positive elements is called an exponential field, an ordered field with a strictly increasing \textit{isomorphism} between the additive group and multiplicative group of positive elements is called a \textit{logarithmic-exponential} field.
\end{remark}

\begin{definition}
    Let $\mathcal K$ be an ordered field, $\mathcal G$ be a multiplicative ordered abelian group. 
    Define $K((\mathcal G))$ as
    $$\{f :G\to K \mid \supp(f) \text{ is conversely well-ordered (i.e. there is the largest element in every nonempty subset)}\},$$ where $\supp(f) := \{g\in G \mid f(g)\neq 0\}$. Elements of $K((\mathcal G))$ will be understood as $\sum\limits_{g\in G} f(g)g$. Also define the structure of an ordered field on $K((\mathcal G))$ as follows:
    \begin{itemize}
        \item $f_1 + f_2$ and $-f$ are defined elementwise;
        \item $f_1 \cdot f_2 := f_3$, where $f_3(g) = \sum\limits_{\substack{g_1, g_2 \in G\\g_1g_2 = g}} f_1(g_1)f_2(g_2)$ (the latter is well-defined since $\supp(f_1)$ and $\supp(f_2)$ are conversely well-ordered);  
        \item $f > 0$ if $\supp(f) \ne \varnothing$ and $f(g_{\max}) > 0$, where $g_{\max} = \max \supp(f)$;
        \item $f_1 > f_2$ if $f_1 - f_2 > 0$. 
    \end{itemize}
\end{definition}

\begin{proposition}[{see, for example, \cite[Chapter VIII, Theorem 10]{fuchs}}]
    $$\mathcal K((\mathcal G)) := (K((\mathcal G)), +, -, \cdot, 0_{\mathcal K}1_\mathcal G, 1_\mathcal K 1_\mathcal G, <)$$ is an ordered field, where $0_{\mathcal K}1_\mathcal G$ and $1_\mathcal K 1_\mathcal G$ are interpretations of 0 and 1 respectively. Moreover, $x \mapsto x 1_\mathcal G$ is an embedding of $\mathcal K$ in $\mathcal K((\mathcal G))$. Here we denote by $x 1_\mathcal G$ for $x \in K$ the function $f\colon  g \mapsto \begin{cases}
        x, \text{if } g = 1_\mathcal G,\\
        0_\mathcal K, \text{if } g \ne 1_\mathcal G.
    \end{cases}$
\end{proposition}

First, we describe the field of exponential series, denoted $\R((t))^E$, which extends the exponential field $\R_{\exp} = (\R, e^x)$. It will be constructed as the union of an increasing sequence of \emph{pre-exponential fields}. 

\begin{definition}
    The quadruple $(\mathcal K, A, B, E)$ is called a \textit{pre-exponential field} if $\mathcal K$ is an ordered field, $A$ is an additive subgroup of $\mathcal K$, $B$ is a convex additive subgroup of $\mathcal K$ (i.e. if $x, y \in B$ and $x < z < y$, then $z \in B$), $A \oplus B = K$, $E$ is an order-preserving embedding from $B$ into the multiplicative group of positive elements of $K$.
\end{definition}

We define a multiplicative ordered abelian group $x^\mathcal \R$ consisting of elements the form $x^r, r \in \R$, with operations defined by $x^r \cdot x^q := x^{r + q}$ and $x^r < x^q :\iff r < q$. Let 
$$A_0 := \{f \in \R((x^\mathcal \R)) \mid \supp(f) > x^0\},$$ 
$$B_0 := \{f \in \R((x^\mathcal \R)) \mid \supp(f) \leqslant x^0\},$$
$$m_0 := \{f \in \R((x^\mathcal \R)) \mid \supp(f) < x^0\}.$$
For $b \in B_0$ there are unique $r \in \R$ and $\varepsilon \in m_0$ such that $b = r + \varepsilon$ (namely, $r = b(x^0)$ and $\varepsilon = b - b(x^0)$). Then let $$E_0(b) := \exp(r)\sum\limits_{n = 0}^{\infty} \frac{\varepsilon^n}{n!}.$$ It is easy to check that the sum $\sum\limits_{n = 0}^{\infty} \frac{\varepsilon^n}{n!}$ is well-defined.

\begin{proposition}[\cite{vandendries_macintyre_marker:2001}]
    $(\mathcal \R((x^\R)), A_0, B_0, E_0)$ is a pre-exponential field.
\end{proposition}

Now we construct an increasing sequence $(\mathcal R_n, A_n, B_n, E_n)_{n \in \N}$ of pre-exponential fields starting from \\$(\mathcal \R((x^\R)), A_0, B_0, E_0)$. Given $(\mathcal R_n, A_n, B_n, E_n)$, construct $(\mathcal R_{n + 1}, A_{n + 1}, B_{n + 1}, E_{n + 1})$ as follows. Let $e^{A_n}$ be a multiplicative ordered group isomorphic to $A_n$, its elements will be denoted as $e^a$ ($a \in A_n$). Then 
$$\mathcal R_{n + 1} := \mathcal R_n((e^{A_n})),$$ 
$$A_{n + 1} := \{f \in R_{n + 1} \mid \supp(f) > e^0\},$$ 
$$B_{n + 1} := \{f \in R_{n + 1} \mid \supp(f) \leqslant e^0\},$$
$$m_{n + 1} := \{f \in R_{n + 1} \mid \supp(f) < e^0\}.$$
It is clear that $\mathcal R_{n + 1} = A_{n + 1}\oplus B_{n + 1}$ and $B_{n + 1} = \mathcal R_n \oplus m_{n + 1}$. For all $b' \in B_{n + 1}$ there are $a \in A_n$, $b \in B_n$ and $\varepsilon \in m_{n + 1}$ such that $b' = a + b + \varepsilon$. Then define 
$$E_{n + 1}(b') := e^a E_n(b) \sum\limits_{n = 0}^{\infty} \frac{\varepsilon^n}{n!}.$$

\begin{proposition}[\cite{vandendries_macintyre_marker:2001}]
    $(\mathcal R_{n + 1}, A_{n + 1}, B_{n + 1}, E_{n + 1})$ is a pre-exponential field, moreover $E_{n + 1} \big|_{B_n} = E_n$. 
\end{proposition}

Define an ordered field $\R((t))^E$ as $\bigcup\limits_{n = 0}^{\infty} \mathcal R_n$ and a map $E \colon  \R((t))^E \to (\R((t))^E)^{>0}$ as $\bigcup\limits_{n = 0}^{\infty} E_n$. Here $t$ stands for $x^{-1}$.

\begin{proposition}[\cite{vandendries_macintyre_marker:2001}]
    $E$ is a strictly increasing homomorphism from $(\R((t))^E, +)$ to $\big((\R((t))^E)^{>0}, \cdot \big)$.
\end{proposition}

Next define $\mathcal G_0 := x^\mathcal \R$ and $\mathcal G_{n + 1} := \mathcal G_n E (A_n)$ for all $n \in \N$. Also define $\mathcal G^E := \bigcup\limits_{n \in \N} \mathcal G_n$. Then the following holds.

\begin{proposition}[{\cite{vandendries_macintyre_marker:2001}}]\label{prop_for_A_n}
For all $n \in \N$, the following holds:
    \begin{enumerate}[(i)]
        \item $G_n \cap E(A_n) = \{1\}$;
        \item $\mathcal R_n \cong \R((\mathcal G_n))$;
        \item $R_n = A_n \oplus \dots \oplus  A_0 \oplus \R \oplus m_0$;
        \item $A_n \oplus \dots \oplus  A_0 = \{ f \in \R((\mathcal G_n)) \mid \supp f > 1\}$;
        \item if $a \in A_{n}$ and $a > 0$, then, for all $m < n$, $a > G_m$.
    \end{enumerate}
\end{proposition}

By \cref{prop_for_A_n} (ii) we may identify $\R((t))^E$ with a subfield of $\R((\mathcal G^E))$ in a unique way. Hence, one may speak about $\supp f$ for $f \in \R((t))^E$. 

Our next aim is to add logarithms to $(\R((t))^E, E)$. Define a map $\Phi\colon  \R((t))^E \to \R((t))^E$ recursively as follows: 
$$\Phi(f) := \begin{cases}
    \sum a_r E(r x), \text{ if } f = \sum a_r x^r \in R_0 =  \R((x^\R)), \\
    \sum \Phi(f_a) E(\Phi(a)), \text{ if } \sum f_a e^a\in R_{n+1} = R_n((e^{A_n})).
\end{cases}$$
Informally speaking, $\Phi$ is a substitution of $E(x)$ for $x$.

Now we define an increasing sequence $(\mathcal L_n, \tilde E_n)_{n = 0}^\infty$ of isomorphic copies of $(\R((t))^E, E)$ with isomorphisms $\eta_n\colon  (\mathcal L_n, \tilde E_n) \to (\R((t))^E, E)$. Let $(\mathcal L_0, \tilde E_0) := (\R((t))^E, E)$, $\eta_0 := \id_{L_0}$. Suppose we have already defined $(\mathcal L_n, \tilde E_n)$ and $\eta_n$. $(\mathcal L_{n+1}, \tilde E_{n + 1})$ is an isomorphic copy of $(\R((t))^E, E)$ with isomorphism $\eta_{n+1}$ such that $L_n \subseteq L_{n+1}$ and for all $z \in L_{n+1}$ we have $\eta_{n+1}(z) = \Phi(\eta_n(z))$. Informally speaking, $\mathcal L_{n + 1}$ is obtained from $\mathcal L_n$ by applying $\Phi^{-1}$, i.e., by substituting $E^{-1}(x)$ for $x$. We have constructed an increasing sequence of fields: 
$$\R((t))^E = (\mathcal L_0, \tilde E_0) \subseteq (\mathcal L_1, \tilde E_1) \subseteq \dots.$$
Now let $\mathcal \R((t))^{\LE} := \bigcup\limits_{n = 0}^\infty (\mathcal L_n, \tilde E_n)$.

\begin{proposition}[\cite{vandendries_macintyre_marker:2001}]
    $\R((t))^{\LE}$ is an exponential field.
\end{proposition}

Define groups $\mathcal G^{E, n} := \eta_n^{-1}(\mathcal G^E)$ for all $n  \in \N$ and $\mathcal G^{\LE} := \bigcup\limits_{n \in \N} \mathcal G^{E, n}$.

\begin{proposition}[\cite{vandendries_macintyre_marker:2001}]
    The sequence $(\mathcal G^{E, n})_{n \in \N}$ is increasing, i.e.
    $$\mathcal G^{E, 0} \subseteq \mathcal G^{E, 1} \subseteq \mathcal G^{E, 2} \subseteq \mathcal \dots.$$
    Hence, the group $\mathcal G^{\LE}$ is well-defined. 
\end{proposition}

Since $\R((t))^E$ is a subfield of $\R((\mathcal G^E))$, $\mathcal L_n$ is a subfield of $\R((\mathcal G^{E, n}))$. So, $\mathcal K((t))^{\LE}$ can be viewed as a subfield of $\mathcal K((\mathcal G^{\LE}))$ and for $f \in \mathcal{K}((t))^{\LE}$ the support $\supp f$ is well-defined (as for elements of $\mathcal K((\mathcal G^{\LE}))$).

It will be no harm to denote exponentiation on $\R((t))^{\LE}$ as $E$. Also denote by $L$ the inverse to $E$.

Denote by $\mathbb{R}_\mathrm{an, exp}$ the ordered field of real numbers expanded with the exponential function and all analytic functions restricted to the cube $[-1, 1]^n$. The following important results hold.

\begin{theorem}[{\cite[Corollary 5.13]{vandendries_macintyre_marker:1994}}]\label{theorem_R_an,exp_is_o-minimal}
    The structure $\mathbb{R}_\mathrm{an, exp}$ is o-minimal.
\end{theorem}

\begin{theorem}[{\cite[Corollary 2.8]{vandendries_macintyre_marker:1997}}]\label{theorem_R((t))^LE_elem_equiv_R_an,exp}
    $\mathbb{R}((t))^{\LE}$ can be expanded to a model of $\Th(\mathbb R_{\mathrm{an,exp}})$. Hence, $\mathbb{R}((t))^{\LE}$ is a model of $\Th(\R_{\exp})$.
\end{theorem}

In our definition of $\mathbb{R}((t))^{\LE}$ we have $E(1) = e$. We can define base-2 exponentiation $E_2$ on $\R((t))^{\LE}$ as $E_2(x) = E(x\ln(2)).$ Then, it follows from \autoref{theorem_R((t))^LE_elem_equiv_R_an,exp} that $(\mathbb{R}((t))^{\LE}, E_2) \vDash \Th(\mathbb R_{\exp_2})$, where $\exp_2(x) = 2^x$.

It remains for us to find an exponential integer part of $\mathbb{R}((t))^{\LE}$. Denote
$$M := \{f \in \mathbb{R}((t))^{\LE} \mid \supp(f) \geqslant x^0 \text{ and the coefficient before $x^0$ lies in }\mathbb{Z}\}$$ 
and denote by $\mathcal M$ the ordered ring with domain $M$. Clearly, it is discretely ordered and is an integer part of $\mathbb{R}((t))^{\LE}$ (see \cite[Lemma 3.2]{mourgues:1993}). Also $\mathcal{M}^+$ is not isomorphic to the standard model, since, for all $n\in \mathbb{N}$, $x > n$, where $x = x^1 \in {M}^+$. Next we show that $M^+$ is closed under $E_2$ and $x^y = E(y L(x))$.

By a \emph{monomial} we mean an element of $\R((t))^{\LE}$ of the form $a g$, $a \in K$, $g \in G^{\LE}$. Clearly, the product of monomials is a monomial. 

\begin{lemma}\label{lemma_for_E(f)}
    Let $f \in \R((t))^{\LE}$. If $\supp f \geqslant x^0$, then $E(f)$ is a monomial. Moreover, if $f$ is infinitely large (i.e., $f > \R$), then $\supp E(f) > x^0$. 
\end{lemma}

\begin{proof}
    Since $(\R((t))^E, E)$ and $(\mathcal{L}_n, \tilde{E}_n)$ are isomorphic, it suffices to prove the statement for $f \in \R((t))^E$. We proceed by induction on $n$ to establish the claim for $f \in R_n$.
    
    First consider the case of $n = 0$. Then $f \in \R((x^\R))$ and $f = f_0 + r$, $\supp f_0 > x^0$ and $r \in \R$  (there is no summand from $m_{0}$ since $\supp f \geqslant x^0$). So, by definition, $E(f) = \exp(r) e^{f_0}$ an the latter is a monomial. Moreover, if $f$ is infinitely large, then $f_0 > 0$ and, hence, $\supp (\exp(r)e^{f_0}) = e^{f_0} > x^0$. 
    
    Now consider inductive step from $n$ to $n + 1$. Then $f = f_0 + f_1$, where $f_0 \in A_{n + 1}$, $f_1 \in R_n$. So, by definition, $E(f) = e^{f_0} E(f_1)$ and $E(f_1)$ is a monomial by the induction hypothesis. Hence, $E(f)$ is a monomial. Additionally, let $f$ be infinitely large. Then either $f_0 = 0$ and $\supp E(f) > x^0$ by the induction hypothesis, or $f_0 > 0$ and $\supp E(f) = \supp (e^{f_0}) \supp (E(f_1)) > x^0$ since $e^{f_0} > G_n$ (\cref{prop_for_A_n}).
\end{proof}

\begin{lemma}\label{lemma_for_exp(eps)}
    For $f \in \R((t))^{\LE}$, if $\supp f < x^0$, then $E(f) = \sum\limits_{i = 0}^\infty \frac{f^i}{i!}$.
\end{lemma}

\begin{proof}
    As in the previous lemma, it suffices to prove the claim for $f \in \R((t))^E$, so we proceed by induction on $n$ to establish the claim for $f \in R_n$.

    For $n = 0$ this is trivial due to the definition. If $f \in R_{n + 1}$, then $f = f_1 + f_2$, where $f_1 \in R_n$ and $f_2 \in m_{n + 1}$ (there is no summand from $A_{n + 1}$ since $\supp f < x^0$). So, by definition and the induction hypothesis, we have 
    $$E(f) = E(f_1) \sum\limits_{j = 0}^\infty \frac{f_2^j}{j!} = \sum\limits_{i = 0}^\infty \frac{f_1^i}{i!} \sum\limits_{j = 0}^\infty \frac{f_2^j}{j!}.$$
    By straightforward computations one can show that $E(f) = \sum\limits_{k = 0}^\infty \frac{(f_1 + f_2)^k}{k!}$.
\end{proof}

\begin{corollary}\label{corollary_for_L(1 + eps)}
    For $f \in \R((t))^{\LE}$, if $\supp f < x^0$, then $L(1 + f) = \sum\limits_{i = 1}^\infty \frac{(-1)^{i + 1} f^i}{i}$.
\end{corollary}

\begin{proof}
    Just apply \cref{lemma_for_E(f)} to $\sum\limits_{i = 1}^\infty \frac{(-1)^{i + 1} f^i}{i}$.
\end{proof}

\begin{proposition}
    If $f_1, f_2 \in M^{> 0}$, then $f_1^{f_2} \in M^{>0}$. 
\end{proposition}

\begin{proof}
    For $f_2 \in \N$ the claim is trivial, so let $f_2$ be nonstandard. Let 
    $$h := \max \supp f_1$$ an $a \in \R^{>0}$ be the coefficient before $h$ in $f_1$. Set $\hat f := f_1 - a h$. We have $$f_1^{f_2} = (a h)^{f_2} E(f_2 L(1 + \frac{\hat f}{a h})),$$
    and $\supp \frac{\hat f}{a h} < x^0$. By \cref{corollary_for_L(1 + eps)}, 
    $$L(1 + \frac{\hat f}{a h}) = \sum\limits_{i > 0}\frac{(-1)^{i + 1}}{i} \Big(\frac{\hat f}{a h} \Big)^i.$$
    Hence, $$\supp (f_2 L(1 + \frac{\hat f}{a h})) \subseteq \{g_1 \dots g_k h^{-k} g' \mid g_1, \dots, g_k \in \supp \hat f, g' \in \supp f_2, k \in \N^{>0}\}.$$
    Write $f_2 L(1 + \frac{\hat f}{a h})$ as a sum $s_1 + s_2$, where $\supp (s_1) \geqslant x^0$ and $\supp (s_2) < x^0$. By \cref{lemma_for_E(f)}, $E(s_1)$ is a monomial and, moreover, $\supp E(s_1) \geqslant x^0$. By \cref{lemma_for_exp(eps)}, we have
    $$E(s_2) = \sum\limits_{n = 0} \frac{s_2^n}{n!},$$
    so, $$\supp E(s_2) \subseteq \{g_1 \dots g_k h^{-k} g'_1 \dots g_m' \mid g_1, \dots, g_k \in \supp \hat f, \ g'_m, \dots, g'_m \in \supp f_2,\ k, m \in \N\}.$$

    Thus every monomial from $f_1^{f_2}$ is of the form 
    $$c a^{f_2} h^{f_2 - k} E(s_1) g_1 \dots g_{k} g'_1 \dots g'_m$$ where $c \in \R$, $g_1, \dots, g_k \in \supp \hat f$, $g'_1 \dots g'_m \in \supp f_2$ (here $a^{f_2}$ and $h^{f_2 - k}$ a monomials with $\supp > x^0$ by \cref{lemma_for_E(f)}). Consequently, $\supp f_1^{f_2} > x^0$ and $f_1^{f_2} \in M^{>0}$.
\end{proof}

So $\mathcal M \subseteq^{\IP}_{\exp} (\mathbb{R}((t))^{\LE}, E_2)$ and $\mathcal M \subseteq^{\IP}_{x^y} ( \mathbb{R}((t))^{\LE}, E)$. By \autoref{theorem_M_exp_ip_ExpField+MaxVal}, $(\mathcal{M}^+, E_2) \vDash \mathsf{IOpen(exp)}$, by \autoref{theorem_(M, x^y)_ip_R_exp}, $(\mathcal{M}^+, x^y) \vDash\mathsf{IOpen(x^y)}$. Now we can obtain some independence results.

\begin{corollary}\label{corollary_independence_result_1}
    $\mathsf{IOpen(x^y)}$ does not prove the irrationality of $\sqrt 2$.
\end{corollary}
    
\begin{proof}
    Since $x$ and $x\sqrt{2}$ lie in $M^+$, $(\mathcal{M}^+, x^y) \vDash \exists x \exists y (x^2 = 2y^2 \wedge x \ne 0 \wedge y \ne 0)$.
\end{proof}

\begin{remark}
    In the argument above the number $2$ can be replaced by an arbitrary natural number.
\end{remark}

\begin{corollary}\label{corollary_independence_result_2}
    For all $n \in \mathbb N$, $n \geqslant 3$, $\mathsf{IOpen(x^y)} \nvdash \neg\exists x \exists y \exists z (x^n + y^n = z^n \wedge x \ne 0 \wedge y \ne 0 \wedge z \ne 0)$.
\end{corollary}

\begin{proof}
    Similar to \cref{corollary_independence_result_1}.
\end{proof}

\begin{remark}
    Of course, the results above can be stated for $\mathsf{IOpen(\exp)}$ as well.
\end{remark}

\section{Open questions and further results}\label{section_further_results}

One can ask whether the opposite statements to \autoref{theorem_M_exp_ip_ExpField+MaxVal} and \autoref{theorem_(M, x^y)_ip_R_exp} hold. The problem with \autoref{theorem_M_exp_ip_ExpField+MaxVal} is that we can have only powers of $2$ as was discussed in \cref{section_introduction}. So, it remains unclear how to embed an arbitrary model of $\mathsf{IOpen(\exp)}$ in an exponential field as an exponential integer part. In the paper \cite{jerabek2024}, E. Jeřábek faced a similar problem when axiomatizing the theory of exponential integer parts in the language $\mathcal L_{\OSR}(\exp)$. He conjectured that this class is not elementary (see the discussion in the last section of \cite{jerabek2024}). The problem with \autoref{theorem_(M, x^y)_ip_R_exp} is that the theory $\mathsf{KTB}$ (a) seems to be too strong and could be made weaker, (b) has an implicit axiomatization. One can try to formalize Khovanskii's proof in some fragment of $\Th(\R_{\exp})$, but this requires more effort. Additionally, his proof uses Sard's Theorem, which is a non-elementary statement, however, we only need a corollary of it, which can be stated in the first-order language (namely, that the set of critical values of a smooth function has an empty interior). This would lead to a simpler theory, but, nevertheless, it is not obvious, how to prove the axioms of it in an exponential field  with an $x^y$-integer part which is a model of $\mathsf{IOpen(x^y)}$.

\begin{question}
    Does the opposite statement to \autoref{theorem_M_exp_ip_ExpField+MaxVal} hold? That is, given a model $(\mathcal M^+, \exp_M) \vDash \mathsf{IOpen}(\exp)$, does there exist an exponential field $(\mathcal R, \exp_R) \vDash \mathsf{ExpField} + \mathsf{MaxVal}(\mathcal L_{\OR}(\exp)) + \exp(1) = 2$ with ${\exp_R|}_{M^+} = \exp_M$ such that $\mathcal M \subseteq_{\exp}^{\IP} (\mathcal R \exp)$?  We also ask the analogous question for \autoref{theorem_(M, x^y)_ip_R_exp}.
\end{question}

Also there are several problems concerning $\mathsf{IOpen}$ that can be stated for $\mathsf{IOpen(\exp)}$ as well. One is an open question on the decidability of the set of Diophantine equations solvable in models of $\mathsf{IOpen}$, or, more generally, of the set of all $\forall$-sentences provable in $\mathsf{IOpen}$. This question was studied extensively, see \cite{wilkie1978, van_den_dries1980, otero1990} and others. Towards the solution of this problem, A. Wilkie obtained the following result.

\begin{theorem}[\cite{wilkie1978}]
    Every discretely ordered $\Z$-semiring can be embedded in a model of $\mathsf{IOpen}$. 
\end{theorem}

A similar question was posed by E. Jeřábek in the very end of \cite{jerabek2024}. 
\begin{question}[\cite{jerabek2024}]
    Is the theory of exponential integers part of RCEF $\forall$-conservative over $\mathsf{IOpen}$ (or, equivalently, over the theory of discretely ordered $\Z$-semirings)?
\end{question}

The theory of exponential integers parts of RCEF (which is denoted in \cite{jerabek2024} as $\mathsf{TEIP}_{2^x}^+$) is axiomatized as $\mathsf{IOpen}$ plus 
\begin{align}
    &\exp(1) = 2,  \quad \exp(x + y) = \exp(x) \exp(y), \quad \exp(x) > x,\\
    &x > 0 \to \exists y (\exp(y) \leqslant x < \exp(y + 1)).
\end{align}
In our future paper we obtain a result towards this question: namely, the theory $\mathsf{IOpen}$ augmented by axioms (1) is $\forall$-conservative over the theory of discretely ordered $\Z$-semirings. However, we are not able to ``add logarithms'' yet. Another question, which was posed by A. Wilkie (private correspondence) and is interesting for us, is on the $\forall$-fragment of the theory $\mathsf{IOpen(\exp)}$. 
\begin{question}
    Is the theory $\mathsf{IOpen(\exp)}$ $\forall$-conservative over $\mathsf{IOpen}$ augmented by axioms (1) (or even over the theory of discretely ordered $\Z$-semirings augmented by (1))?
\end{question}
An affirmative answer to this question would lead to the $\forall$-conservativity of $\mathsf{IOpen(\exp)}$ over the theory of discretely ordered $\Z$-semirings and hence to a positive solution to the Jeřábek's question.

Another question concerns the theory considered in \autoref{theorem_M_exp_ip_ExpField+MaxVal}, namely, $\mathsf{ExpField} + \mathsf{MaxVal}(\mathcal L_{\OR}(\exp)) + \exp(1) = 2$. While it is known that the theory $\ExpField + \mathsf{RCF} + \forall x (\exp(x) \geqslant 1 + x)$ from \autoref{theorem_(M,x^y)_is_iopen+T_x^y<=>ip_of_ExpField_RCF+Bern)} does not axiomatize $\Th(\R, \exp)$ (see \cite{Dahn1983}), the question whether $\mathsf{ExpField} + \mathsf{MaxVal}(\mathcal L_{\OR}(\exp)) + \exp(1) = 2$ axiomatizes  $\Th(\R, \exp_2)$, where $\exp_2$ denotes the usual base-2 exponentiation, seems to be open. An affirmative answer would imply the decidability of $\Th(\R, \exp_2)$. One may also ask whether $\mathsf{ExpField} + \mathsf{MaxVal}(\mathcal L_{\OR}(\exp)) + \exp(1) = 2$ axiomatizes the possibly weaker theory of definably complete base-2 exponential fields; this question is also open.

\begin{question}
    Does  $\mathsf{ExpField} + \mathsf{MaxVal}(\mathcal L_{\OR}(\exp)) + \exp(1) = 2$ axiomatize the theory of definably complete base-2 exponential fields, or even $\Th(\R, \exp_2)$?
\end{question}

Finally, an important question --- as highlighted in the introduction --- is whether Tennenbaum's Theorem holds for $\mathsf{IOpen(\exp)}$ and $\mathsf{IOpen(x^y)}$, which remains open. 
\begin{question}
    Do $\mathsf{IOpen(\exp)}$ and $\mathsf{IOpen(x^y)}$ have nonstandard recursive models?
\end{question}
Shepherdson's result \cite{shepherdson:1964} shows that for $\mathsf{IOpen}$ the answer is negative. His proof relies on a concrete construction of a non-archimedean real closed field, using Puiseux series, and extracting an integer part from it in a simple way. There are some generalizations of this, for instance, A. Berarducci and M. Otero showed that there is a recursive nonstandard model of normal open induction \cite{berarducci}. Further results on Tennenbaum phenomenon can be found in \cite{daquino97, yaegasi08}. However, adapting Shepherdson’s method to the exponential case is non-trivial, as the construction in \cite{vandendries_macintyre_marker:2001} is non-recursive. Applying some results from recursive model theory, modulo Schanuel's Conjecture, one can obtain a nonstandard model of $\R_{\exp}$ (and, in fact, an elementary recursive submodel of $(\mathbb{R}((t))^{\LE}, E_2)$), but such a model seems to have no ``constructive'' integer part.

\section*{Funding}
This work was performed at the Steklov International Mathematical Center and supported by the Ministry of Science and Higher Education of the Russian Federation (agreement no. 075-15-2025-303).

\bibliographystyle{unsrt}
\bibliography{references}

\end{document}